\documentclass{amsart}
\usepackage{amssymb}
\usepackage{eucal}
\usepackage{amsfonts}
\usepackage{epsfig}
\usepackage{color}
\usepackage[frame,ps,dvips,matrix,arrow,curve,rotate]{xy}

\let\oldcite\cite                                  
\newtheorem{thm}{Theorem}[section]
\newtheorem{cor}[thm]{Corollary}
\newtheorem{lem}[thm]{Lemma}
\newtheorem{prop}[thm]{Proposition}
\theoremstyle{definition}

\theoremstyle{remark}
\newtheorem{rem}[thm]{Remark}
\numberwithin{equation}{section} \theoremstyle{remark}

\newcommand{\bbX}{\mathbb{X}}
\newcommand{\bbY}{\mathbb{Y}}
\newcommand{\bbZ}{\mathbb{Z}}
\newcommand{\bbE}{\mathbb{E}}
\newcommand{\bbF}{\mathbb{F}}

\newcommand{\bbH}{\mathbb{H}}

\newcommand{\C}{\mathsf{C}}
\newcommand{\B}{\mathsf{B}}
\newcommand{\A}{\mathsf{A}}
\newcommand{\sfK}{\mathsf{K}}
\newcommand{\sfF}{\mathsf{F}}
\newcommand{\sfFF}{\mathsf{FF}}
\newcommand{\sfT}{\mathsf{T}}
\newcommand{\sfZT}{\mathsf{ZT}}

\newcommand{\bfX}{\mathbf{X}}
\newcommand{\bfY}{\mathbf{Y}}
\newcommand{\bfZ}{\mathbf{Z}}
\newcommand{\bfE}{\mathbf{E}}
\newcommand{\bfF}{\mathbf{F}}

\newcommand{\mfG}{\mathfrak{G}}

\newcommand{\mcH}{\mathcal{H}}
\newcommand{\mcS}{\mathcal{S}}
\newcommand{\mcT}{\mathcal{T}}
\newcommand{\mcF}{\mathcal{F}}
\newcommand{\mcD}{\mathcal{D}}

\newcommand{\sis}{\mcS_{sis}}
\newcommand{\qis}{\mcS_{qis}}

\newcommand{\im}{\operatorname{im}}
\newcommand{\coker}{\operatorname{coker}}
\newcommand{\Ob}{\operatorname{Ob}}

\newcommand{\Hom}{\operatorname{Hom}}

\newcommand{\homf}{\mathfrak{H}om}
\newcommand{\homfh}{\mathfrak{H}om}         
\newcommand{\homfone}{\mathcal{M}}
\newcommand{\RHom}{\mathfrak{RH}om}

\newcommand{\Totfrak}{\mathfrak{Tot}}

\newcommand{\Ch}{\mathsf{Ch}}
\newcommand{\Chfrak}{\mathfrak{Ch}}
\newcommand{\Chst}{\mathsf{Ch}^{st}}
\newcommand{\Chstfrak}{\mathfrak{Ch}^{st}}
\newcommand{\Dec}{\mathsf{Dec}}
\newcommand{\Decfrak}{\mathfrak{Dec}}

\newcommand{\KDec}{\mathsf{KDec}}

\newcommand{\Tot}{\operatorname{Tot}}

\newcommand{\Cyl}{\operatorname{Cyl}}
\newcommand{\Cone}{\operatorname{Cone}}

\newcommand{\id}{\operatorname{id}}
\newcommand{\pr}{\operatorname{pr}}
\newcommand{\inc}{\operatorname{in}}

\newcommand{\taudl}[1]{\tau_{\leq #1}}
\newcommand{\taudg}[1]{\tau_{\geq #1}}

\let\:=\colon

\def\smashedlongrightarrow{\setbox0=\hbox{$\longrightarrow$}\ht0=1pt\box0}
\def\risom{\buildrel\sim\over{\smashedlongrightarrow}}

\newcommand{\lra}{\longrightarrow}
\newcommand{\llra}[1]{\stackrel{#1}{\lra}}

\newcommand{\torel}[1]{\stackrel{#1}{\to}}

\newcommand{\Eb}{E^{\bullet}}
\newcommand{\Fb}{F^{\bullet}}

\newcommand{\Mb}{M^{\bullet}}
\newcommand{\Nb}{N^{\bullet}}

\newcommand{\oux}[2]{\underset{#1}{\overset{#2}\oplus}}

\begin{document}

\title{Explicit HRS-tilting}%

\author{Behrang Noohi}
\address{Florida State University, Department of Mathematics,
Tallahassee, Florida 32306-4510, USA}
\email{behrang@alum.mit.edu}
\urladdr{http://www.math.fsu.edu/~noohi/}

 \maketitle

\begin{abstract}
  For an abelian category $\A$ equipped with a torsion pair,
  we give an explicit description for the tilted abelian category $\B$
  introduced in \oldcite{HaReSm}, and also for the categories $\Ch(\B)$
  and $\mcD(\B)$. We also describe the DG structure on $\Ch(\B)$.
  As a consequence, we find new proofs of certain results of
  [ibid.]. The main ingredient is the category of  {\em
  decorated} complexes.
\end{abstract}

\section{Introduction}{\label{S:intro}}

Tilting theory originated from representation theory of (finite
dimensional) algebras and their derived categories (see for instance
\cite{BeGePa,BrBu} for the origins of the theory). It was
essentially conceived as a machinery to compare derived categories
of various algebras. The theory has developed substantially in the
past three decades thanks to the works of various authors such as
Auslander, Happel, Keller, Krause, Reiten, Rickard, Ringel,\dots.

 Nowadays techniques of tilting theory have
found applications in (derived) geometry of varieties,
noncommutative geometry, representation theory (of finite groups,
algebraic groups, quantum groups, quivers, \dots), cluster algebras,
and so on.

A precursor to the introduction of the tilting techniques in
geometry is the work of Beilinson relating the derived category of
coherent sheaves on a projective space to the derived category of a
certain finite dimensional noncommutative algebra \cite{Beilinson}.
This was further developed by Bondal \cite{Bondal} and has now
become a standard tool in the study of derived categories of
varieties.

Tilting theory is also closely related to Bridgeland's theory of
stability conditions \cite{Br}. Let us say a few words on this.
Given a stability condition $(Z, \mathcal{P})$ on a triangulated
category $\mathcal{D}$, the slicing $\mathcal{P}$ gives rise to
abelian categories
$\A_{\theta}:=\mathcal{P}\big((\theta,\theta+1]\big)$ inside
$\mathcal{D}$. It is easy to see that for $0<\theta\leq 1$,
$\A_{\theta}$ is obtained by tilting $\A_0$ with respect to the
torsion pair $(\mathcal{T}_{\theta},\mathcal{F}_{\theta})$, where
$\mathcal{F}_{\theta}:=\mathcal{P}\big((0,\theta]\big)$ and
$\mathcal{T}_{\theta}:=\mathcal{P}\big((\theta,1]\big)$. This
observation has interesting implications in noncommutative geometry.
For example, Polishchuk \cite{Po1,Po2,PoSch} shows that if we apply
this to the derived category $\mathcal{D}(T)$ of coherent sheaves on
a complex torus $T$, with the stability condition being the one
coming from the Harder-Narasimhan filtration, the tilted abelian
category $\A_{\theta}$ will be equivalent to the category of
coherent sheaves on the noncommutative torus $T_{\theta}$.

For more on the relation between tilting theory and stability
conditions on varieties the reader can consult works of Bridgeland
and references therein (e.g., \cite{Br2}).
 An application of tilting theory in noncommutative algebraic
geometry appears in \cite{vdB}.  Applications to perverse sheaves
and representation theory of Lie and quantum groups can be found in
various articles by (one or more) of the authors Beilinson,
Bezrukavnikov, Mirkovic, \dots.

One of the main tools in the works alluded above is the construction
of the `tilting'  of an abelian category with respect to a {\em
torsion pair} \cite{HaReSm}. In [ibid.] the authors associate to an
abelian category $\A$ equipped with a torsion pair $(\mcT,\mcF)$ a
new abelian category $\B$ (which is in turn equipped with its own
torsion pair $(\mcT',\mcF')$). This is the `HRS-tilting' of $\A$.

The construction of $\B$ in [ibid.] is indirect and is carried out
by taking the heart of a certain $t$-structure (associated to the
torsion pair) on the derived category of $\A$. In these notes we
give an alternative construction for $\B$ that is more explicit and
reveals more of the structure of $\B$, as we explain shortly. We
expect this new description to be suitable for geometric
applications, as the explicit nature of $\B$ lends itself well to
geometric manipulations (say when working with bundles over a
variety). It could very well give a new insight to the category
theoretic properties (say, existence of generators, chain
conditions, limits and colomits, etc.) of $\B$ as well.

The main input in  this work is an alternative description of
morphisms in the derived category $\mcD(\A)$ between complexes
concentrated in degrees $[-1,0]$; see (\cite{Maps}, Section 9) and
$\S$\ref{SS:definition}. We exploit this  to give an explicit
description of the category $\Ch(\B)$ of chain complexes in $\B$,
its DG structure (Sections \ref{S:DG} and \ref{S:DGEquiv},
especially, Theorem \ref{T:derequiv}), and its derived category
(Theorem \ref{T:derived}). This is achieved via what we call a {\em
decorated complex}, which might be a notion of independent interest;
see Section \ref{S:decorated}. The correspondence between the
homological algebra of $\B$ and that of decorated complexes in $\A$
is established via a functor $\Tot$ which should be thought of as a
``twisted'' total complexes functor.

Although it is not the main purpose of the paper,  we also show how
our approach leads to new proofs for some of the main results of
Happel-Reiten-Smal\o; see Theorem \ref{T:HRS} and Theorem
\ref{T:HRS2}.

\medskip

\noindent{\bf Outline of the main results}

\medskip

Let $(\mcT,\mcF)$ be a torsion pair on an abelian category $\A$. We
begin by observing that, by results of \cite{Maps}, the abelian
category obtained by performing HRS-tilting on this torsion pair is
equivalent to the following category:
\begin{itemize}
  \item[$\diamond$] $\Ob(\B)=\{ \bbX=[X^{-1}\torel{d}X^0] \ | \ \ker d \in \mcF,
  \coker d \in \mcT\}$.

  \item[$\diamond$] $\Hom_{\B}(\bbX,\bbY)=$ isomorphism classes of
    commutative diagrams
     $$\xymatrix@C=8pt@R=6pt@M=6pt{ X^{-1} \ar[rd]^{\kappa} \ar[dd]_{d}
                            & & Y^{-1} \ar[ld]_{\iota} \ar[dd]^{d} \\
                  & E \ar[ld]^{\sigma} \ar[rd]_{\rho}  & \\
                  X^0 & & Y^0       }$$
    such that the diagonal maps compose to zero and the NE-SW sequence
    is short exact.
\end{itemize}

We use this description to get explicit information about $\B$. For
instance, the kernel and cokernel of a `butterfly' diagram $P$ as
above are given by
  $$\ker P:= [X^{-1} \torel{\kappa} A],$$
  $$\coker P :=  [E/A  \torel{\rho} Y^0].$$
Here, $A$ is the (unique) subobject of $E$ sitting between
$\im\kappa$ and $\ker\rho$ such that $A/\im(\kappa) \in \mcT$  and
$\ker(\rho)/A \in \mcF$. Using this we find a description of
complexes in $\B$; see $\S$\ref{SS:complexes1}.

We then exploit these results to give a description of the derived
category of $\B$ in terms of {\em decorated complexes}
($\S$\ref{S:decorated}). A decorated complex in $\A$ consists of a
complex $E^{\bullet}$ in $A$, together with a collection of
subobjects $M^n \subseteq C^n$, for every $n$. The differentials of
$E^{\bullet}$ are not required to respect the subobjects. A morphism
 $(E^{\bullet},M^{\bullet}) \to
(F^{\bullet},N^{\bullet})$ of decorated complexes  is, by
definition, a chain map $f \: E^{\bullet} \to F^{\bullet}$ which
respects the subobjects. We say that such a morphism is a
quasi-isomorphism if $f$ is so.

A decorated complex $(E^{\bullet},M^{\bullet})$ is said to be
compatible with a torsion pair $(\mcT,\mcF)$ if
\begin{itemize}
  \item[$\blacktriangleright$] \ $M^n\cap\partial^{-1}M^{n+1} \in \mcF$
    and $E^n/(M^n+\partial M^{n-1}) \in \mcT$,  for every $n$.
\end{itemize}
The decorated complexes whose decoration is compatible with the
torsion pair form a full subcategory of $\Dec(\A)$ which we denote
by $\Ch(\A,\mcT,\mcF)$. We have the following description of the
derived category $\mcD(\B)$ of $\B$ (see Corollary \ref{C:main}).

\begin{thm}
  There is a natural equivalence of triangulated categories
   $$\Ch(\A,\mcT,\mcF)/_{q-iso} \cong \mcD(\B).$$
  The same this is true for bounded (above, below, both-sided)
  derived categories.
\end{thm}

The category $\Dec(\A)$ of decorated complexes in $\A$ behaves very
much like the category $\Ch(\A)$ of chain complexes in that we can
define decorated cylinders, cones, homotopies, and so on. In other
words, we can do homological algebra in $\Dec(\A)$. In particular,
we can talk about the homotopy and the derived categories of
$\Dec(\A)$, and these are both triangulated categories.

In the case where $\A$ is the category of $K$-modules for a ring
$K$, $\Dec(K)$ is a closed monoidal category. More generally, if
$\A$ is $K$-linear, then $\Dec(\A)$ is enriched over $\Dec(K)$. This
way $\Ch(\A,\mcT,\mcF)$ inherits a DG structure from $\Dec(\A)$,
which we denote by $\Chfrak(\A,\mcT,\mcF)$. Also, $\Ch(\B)$ has a DG
structure, which we denote by $\Chfrak^*(\B)$.

The above theorem can now be enhanced to a derived equivalence of
derived categories; see Theorem \ref{T:derequiv}.

\begin{thm}
  Assume that $\A$ has either enough injectives or enough
  projectives. Assume further that $\B$ has enough injectives
  (respectively, enough projectives). Let $*=+,b$ (respectively,
  $*=-,b$.) Then, we have a derived equivalence
  $$\Chfrak^*(\A,\mcT,\mcF)\cong\Chfrak^*(\B)$$
  of DG categories.
\end{thm}

Finally, let us remark that, in view of the above results, the
functor $D(\B)\to D(\A)$ (and the DG functor $\Chfrak(\B) \to
\Chfrak(\A)$) studied in \cite{HaReSm} is nothing but the forgetful
functor
 $$(E^{\bullet},M^{\bullet}) \mapsto E^{\bullet}$$
that forgets the decoration. In the case where the torsion pair is
tilting or cotilting this is known to be an (derived) equivalence;
see Theorem \ref{T:HRS}.


 \tableofcontents

\section{A quick review of torsion theories}{\label{S:Torsion}}

Let $\A$ be an abelian category.
 A {\bf torsion theory} in $\A$ is a
pair $(\mcT,\mcF)$ of full additive subcategories of $\A$ such that:

\vspace{0.1in}

\begin{itemize}
  \item[$\blacktriangleright$]  For every $T \in \mcT$ and $F\in \mcF$,
   we have $\Hom(T,F)=0$.
   \vspace{0.1in}
  \item[$\blacktriangleright$] For every $A \in \A$, there is a
   (necessarily unique) exact sequence
    $$0 \to T \to A \to F \to 0, \ \ \ \ T\in \mcT,\ F\in\mcF.$$
\end{itemize}

The following facts are well-known and easy to prove.

\begin{lem}{\label{L:torsion1}}
  For a torsion theory $(\mcT,\mcF)$ we have $\mcT^{\bot}=\mcF$ and
  $^{\bot}\mcF=\mcT$, that is
   $$\mcF=\{F\in \A \ | \ \forall T\in \mcT, \ \Hom(T,F)=0\},$$
   $$\mcT=\{T\in \A \ | \ \forall F\in \mcF, \ \Hom(T,F)=0\}.$$

\end{lem}

\begin{lem}{\label{L:torsion2}}
  If $X \to Y$ is a monomorphism and $Y$ is in $\mcF$, then $X$ is
  in $\mcF$. If $Y \to Z$ is an epimorphism and $Y$ is in $\mcT$,
  then $Z$ is in $\mcT$.
\end{lem}

Remark that it is not true in general that a subobject of an object
$Y \in \mcT$ is  in $\mcT$. Similarly, it is not true in general
that a quotient of an object $Y \in \mcF$ is in $\mcF$.

\begin{lem}{\label{L:torsion3}}
  Consider the exact sequence
    $$0\to X \to Y \to Z \to 0$$
  in $\A$. If $X$ and $Z$ are both in $\mcT$ (respectively, in $\mcF$),
  then so is $Y$.
\end{lem}

\section{The category $\B$}{\label{S:B}}

 Let $\A$ be an abelian category and $(\mcT,\mcF)$
 a torsion pair in $\A$. To this data we associate a
 new abelian category $\B$ and a
torsion pair $(\mcT',\mcF')$. 
By results of (\cite{Maps}, Section 9) this  category is naturally
equivalent to the one defined in \cite{HaReSm}.

\subsection{Definition of $\B$}{\label{SS:definition}}

The category $\B$ is defined as follows:

\begin{itemize}
  \item[$\diamond$] $\Ob(\B)=\{ \bbX=[X^{-1}\torel{d}X^0] \ | \ \ker d \in \mcF,
  \coker d \in \mcT\}$. We will usually drop $d$ from the notation.

  \item[$\diamond$] $\Hom_{\B}(\bbX,\bbY)=$ isomorphism classes of
    commutative diagrams
     $$\xymatrix@C=8pt@R=6pt@M=6pt{ X^{-1} \ar[rd]^{\kappa} \ar[dd]_{d}
                            & & Y^{-1} \ar[ld]_{\iota} \ar[dd]^{d} \\
                  & E \ar[ld]^{\sigma} \ar[rd]_{\rho}  & \\
                  X^0 & & Y^0       }$$
    such that the diagonal maps compose to zero and the NE-SW sequence
    is short exact.

\end{itemize}

\begin{rem}
 It follows from the axioms of a torsion pair that, given objects 
 $\bbX$ and $\bbY$ in $\B$
 and two diagrams $E$ and $E'$ as above, there exists at
 most one isomorphism between $E \to E'$ commuting with all the 
 four arrows of the two diagrams.
 Therefore, by passing to isomorphism classes of such
 diagrams we do not loose any information.
\end{rem}

A morphism that comes from  an actual morphism of complexes $f\:
\bbX \to \bbY$ in $\Ch(\A)$ corresponds to the diagram
        $$\xymatrix@C=8pt@R=6pt@M=6pt{ X^{-1} \ar[rd]^{(d,-f^{-1})}  \ar[dd]_{d}
                            & & Y^{-1} \ar[ld]  \ar[dd]^{d} \\
                  &  X^0\oplus Y^{-1} \ar[ld]  \ar[rd]_{f^0+d}   & \\
                  X^0 & & Y^0       }$$
 For simplicity, we denote such morphisms in the usual way
    $$\xymatrix@C=8pt@R=6pt@M=6pt{ X^{-1}  \ar[rr]^{f^{-1}}  \ar[dd]_{d}
                            & & Y^{-1}   \ar[dd]^{d} \\
                  &          & \\
                  X^0 \ar[rr]_{f^0} & & Y^0       }$$
and call them {\bf strict} morphisms. Equivalently, a strict
morphism in $\B$ is one for which the NE-SW sequence splits.

\begin{lem}{\label{L:split}}
  If $\bbX$ is such that  $X^0$ is projective,
  then every morphism coming out of $\bbX$ is strict.
  If $\bbY$ is such that  $Y^{-1}$ is injective,
  then every morphism to $\bbY$ is strict.
\end{lem}

\begin{proof}
  Trivial.
\end{proof}

\subsection{Composition of morphisms}{\label{SS:compose}}

Given two morphisms

  $$ \xymatrix@C=8pt@R=6pt@M=6pt{ X^{-1} \ar[rd]  \ar[dd]
                            & & Y^{-1} \ar[ld]_{\iota}  \ar[dd]  \\
                  & E \ar[ld]  \ar[rd]_{\rho}  & \\
                  X^0 & & Y^0       }\ \ \ \ \
  \xymatrix@C=8pt@R=6pt@M=6pt{ Y^{-1} \ar[rd]^{\kappa'}  \ar[dd]
                            & & Z^{-1} \ar[ld]  \ar[dd]  \\
                  & F \ar[ld]^{\sigma'}  \ar[rd]   & \\
                  Y^0 & & Z^0       }$$
in $\B$, we define their composition  to be
$$\xymatrix@C=10pt@R=4pt@M=4pt{ X^{-1} \ar[rd]  \ar[dd]
                            & & Z^{-1} \ar[ld]  \ar[dd]  \\
                  & E\oux{Y^0}{Y^{-1}}F \ar[ld]  \ar[rd]   & \\
                  X^0 & & Z^0       }$$
Here, $E\oux{Y^0}{Y^{-1}}F$ is the quotient of the object $L$
consisting of pairs
  $(x,y) \in E\times F$  such that $\rho(x)=\sigma'(y) \in Y^0$,
  modulo the subobject $I=\{\big(\iota(\beta),\kappa'(\beta)\big)
  \in E\times F\ | \ \beta \in Y^{-1}\}$. More precisely, let
  $E\oux{Y^0}{}F$ be the fiber product of $E$ and $F$ over
  $Y^{0}$. Then
    $$E\oux{Y^0}{Y^{-1}}F:=\coker (Y^{-1} \llra{(\iota,\kappa')} E\oux{Y^0}{}F).$$

In the case where one of the morphisms is strict, the composition
takes a simpler form. When the first morphisms is strict, say
          $$\xymatrix@C=8pt@R=6pt@M=6pt{ X^{-1}  \ar[rr]^{f^{-1}}
          \ar[dd]
                            & & Y^{-1}   \ar[dd]  \\
                  &          & \\
                  X^0 \ar[rr]_{f^0} & & Y^0       }$$
 then the composition is
         $$   \xymatrix@C=8pt@R=6pt@M=6pt{ X^{-1} \ar[rd]   \ar[dd]
                            & & Z^{-1} \ar[ld]  \ar[dd]  \\
                  & f^{0,*}(F) \ar[ld]^{f^{0,*}(\sigma')}  \ar[rd]   & \\
                  X^0 & & Z^0       }$$
Here, $f^{0,*}(F)$ stands for the pull back of the extension $F$
along $f^0 \: X^0 \to Y^0$. More precisely,
$f^{0,*}(F)=X^0\oplus_{Y^0}F$ is the fiber product.

When the second morphisms is strict, say
          $$\xymatrix@C=8pt@R=6pt@M=6pt{ Y^{-1}  \ar[rr]^{g^{-1}}
          \ar[dd]
                            & & Z^{-1}   \ar[dd]  \\
                  &          & \\
                  Y^0 \ar[rr]_{g^0} & & Z^0       }$$
 then the composition is
         $$   \xymatrix@C=8pt@R=6pt@M=6pt{ X^{-1} \ar[rd]   \ar[dd]
                            & & Z^{-1} \ar[ld]_{g^{-1}_*(\iota)}  \ar[dd]  \\
                  & g^{-1}_*(E) \ar[ld]  \ar[rd]   & \\
                  X^0 & & Z^0       }$$
Here, $g^{-1}_*(E)$ stands for the push forward of the extension $E$
along $g^{-1} \: Y^{-1} \to Z^{-1}$. More precisely,
$g^{-1}_*(E)=E\oplus^{Y^{-1}} Z^{-1}$ is the push-out.

\subsection{Addition  of morphisms}{\label{SS:addition}}

  Given two elements $P,P' \in \Hom(\bbX,\bbY)$,
       $$\xymatrix@C=8pt@R=6pt@M=6pt{ X^{-1} \ar[rd]^{\kappa}
       \ar[dd]
                            & & Y^{-1} \ar[ld]_{\iota} \ar[dd]  \\
                  & E \ar[ld]^{\sigma} \ar[rd]_{\rho}  & \\
                  X^0 & & Y^0       } \ \   \ \
       \xymatrix@C=8pt@R=6pt@M=6pt{ X^{-1} \ar[rd]^{\kappa'} \ar[dd]
                            & & Y^{-1} \ar[ld]_{\iota'} \ar[dd] \\
                  & E' \ar[ld]^{\sigma'} \ar[rd]_{\rho'}  & \\
                  X^0 & & Y^0       }$$
we define $P+P' \in \Hom(\bbX,\bbY)$ to be
   $$\xymatrix@C=8pt@R=6pt@M=6pt{ X^{-1} \ar[rd]^{(\kappa,\kappa')}
   \ar[dd]
                            & & Y^{-1} \ar[ld]_{(0,\iota)} \ar[dd] \\
                  &
E\oux{X^0}{Y^{-1}}E' \ar[ld]^{\sigma=\sigma'} \ar[rd] _{\rho+\rho'}  & \\
                  X^0 & & Y^0       }$$
where $E\oux{X^0}{Y^{-1}}E'$ is defined as in $\S$\ref{SS:compose},
with the difference that now we mod out  $E\oux{X^0}{}E'$  by the
antidiagonal image of $Y^{-1}$ instead of the diagonal image. The
map in the bottom-left corner, denoted by $\sigma=\sigma'$, sends
$(a,b)$ to $\sigma(a)$ (which, by definition, is equal to
$\sigma'(b)$).

We define $-P$ by
 $$\xymatrix@C=8pt@R=6pt@M=6pt{ X^{-1} \ar[rd]^{\kappa} \ar[dd]_{d}
                            & & Y^{-1} \ar[ld]_{-\iota} \ar[dd]^{d} \\
                  & E \ar[ld]^{\sigma} \ar[rd]_{-\rho}  & \\
                  X^0 & & Y^0       }$$

  When $\A$ is $R$-linear for some commutative ring $R$, then
  $\B$ is also naturally $R$-linear. For $r \in R$  and $P$
  as above, $rP$  is equal to the composition
  of $P$ and the strict morphism $r \cdot - : \bbY \to \bbY$;
  see the end of $\S$\ref{SS:compose} to see what this exactly is.
  In the case where $r$ is a unit, $rP$ is represented by
  $(E,\kappa,r^{-1}\iota,\sigma,r\rho)$.

\subsection{Kernels, cokernels}{\label{SS:kernel}}

Consider $P \in \Hom(\bbX,\bbY)$ given by
   $$\xymatrix@C=8pt@R=6pt@M=6pt{ X^{-1} \ar[rd]^{\kappa} \ar[dd]
                            & & Y^{-1} \ar[ld]_{\iota} \ar[dd]  \\
                  & E \ar[ld]^{\sigma} \ar[rd]_{\rho}  & \\
                  X^0 & & Y^0       }$$
The cone of $P$, where now we consider $P$ as a morphism in the
derived category $\mcD(\A)$, has a natural model, namely, the NW-SE
complex
   $$C(P):=X^{-1} \torel{\kappa} E \torel{\rho} Y^0,$$
in which $Y^0$ is sitting in degree $0$. The corresponding triangle
$$\bbX \to \bbY \to C(P) \to \bbX[1]$$
is defined in the obvious way.

From this we get the following descriptions of the kernel and
cokernel of $P$ in $\B$.

\vspace{0.1in}

\noindent {\bf Kernel.} Let $A=q^{-1}(T)$, where $T \in \mcT$ is the
torsion part of  $H^{-1}(C(P))$ and $q \: \ker\rho \to H^{-1}(C(P))$
is the quotient map. Then, the kernel of $P$ is
     $$\ker P:= [X^{-1} \torel{\kappa} A].$$
The map $\ker P \to \bbX$ is given by $(\id_{X^{-1}},\sigma|_A)$. We
have $H^{-1}(\ker(P))=H^{-2}(C(P))$ and $H^{0}(\ker(P))=T$.

\vspace{0.1in}

\noindent {\bf Cokernel.} 
 The cokernel of $P$ is
      $$\coker P :=  [E/A  \torel{\rho} Y^0].$$
The map $\bbY \to \coker P$ is given by $(\iota,\id_{Y^0})$. We have
$H^{-1}(\coker(P))=F$, the free part of  $H^{-1}(C(P))$, and
$H^{0}(\coker(P))=H^0(C(P))$

\begin{cor}{\label{C:monoepi}}
  A morphism $P$ as above is a monomorphism if and only if $\kappa$ is a
  monomorphism and $H^{-1}(C(P))\in \mcF$. The morphism $P$ is an epimorphism
  if and only if $\rho$ is an
  epimorphism and  $H^{-1}(C(P)) \in \mcT$.
\end{cor}

\begin{cor}{\label{C:isom}}
   A morphism $P$ as above is an isomorphism  if and only if
   the NW-SE sequence $C(P)$ is short exact. In this case, the inverse of
   $P$ is obtained by flipping the diagram with respect to the vertical axis.
\end{cor}

\begin{cor}
  Let $f\: \bbX \to \bbY$ be a strict morphism that is an
  equivalence. Then, the inverse $f^{-1} \: \bbY \to \bbX$
  corresponds to the diagram

        $$\xymatrix@C=8pt@R=6pt@M=6pt{ Y^{-1} \ar[rd]^{\inc_2}  \ar[dd]_{d}
                            & & X^{-1} \ar[ld]_{(d,-f^{-1})}  \ar[dd]^{d} \\
                  &  X^0 \oplus Y^{-1} \ar[ld]^{f^0+d}  \ar[rd]_{\pr_1}   & \\
                  Y^0 & & X^0       }$$
\end{cor}

\begin{proof}
 Use the discussion of $\S$\ref{SS:definition} to find the
 diagram corresponding to $f$. Then flip the diagram.
\end{proof}

The short exact sequence
  $$0\to \ker P[1] \to C(P) \to \coker(P)\to 0$$
 of complexes  gives rise to the  exact sequence
  $$0\to H^0(\ker P) \to H^{-1}(C(P)) \to H^{-1}(\coker P) \to 0$$
of cohomologies.

We also have the following.

\begin{prop}{\label{P:long}}
  There is a long exact sequence
   {\footnotesize
  $$
    0 \to H^{-2}(C(P)) \to  H^{-1}(\bbX) \to H^{-1}(\bbY)  \to
         H^{-1}(C(P)) \to H^0(\bbX) \to H^0(\bbY) \to H^0(C(P)) \to 0.
  $$
  }
\end{prop}

\begin{proof}
This is the  exact sequence for the exact triangle $\bbX \to \bbY
\to C(P) \to \bbX[1]$.
\end{proof}

\subsection{The epi-mono factorization}{\label{SS:epimono}}

Notation being as in $\S$\ref{SS:kernel}, it is easy to see that the
cokernel of the map $\ker P \to \bbX$ is the complex
  $$\operatorname{coim} P:=[A \torel{\sigma} X^0],$$
and the kernel of $\bbY \to \coker P$ is the complex
  $$\im P:=[Y^{-1} \llra{\pr\circ\iota} E/A].$$

There is a canonical isomorphism $\operatorname{coim} P \to \im P$
given by
  $$\xymatrix@C=8pt@R=6pt@M=6pt{ A \ar[rd] \ar[dd]
                            & & Y^{-1} \ar[ld]_{\iota} \ar[dd] \\
                  & E \ar[ld]^{\sigma} \ar[rd]  & \\
                  X^0 & & E/A       }$$
(See Corollary \ref{C:isom}.) So the epi-mono factorization of $P$
looks like
 $$\xymatrix@C=8pt@R=6pt@M=6pt{X^{-1} \ar[r]^{\kappa} \ar[dd]  & A \ar[rd] \ar[dd]
          & & Y^{-1} \ar[ld]_{\iota} \ar[dd] \ar[r]^{\id} & Y^{-1} \ar[dd]  & \\
                            & & E \ar[ld]^{\sigma} \ar[rd]  &  &\\
                  X^0 \ar[r]_{\id} &  X^0 & & E/A \ar[r]_{\rho} & Y^0      }$$

\section{Complexes in $\B$ and strict morphisms between them}
   {\label{S:complexesstrict}}

In this section we prepare ourselves for the first main result of
these notes that will appear in Section \ref{S:complexes}; see
$\S$\ref{SS:main}. One of the major players here is the category
$\Chst(\B)$ defined below.

Let $\Chst(\B) \subset \Ch(\B)$ be the category whose objects are
complexes
     $$\bfX= \\\cdots \to\, ^{n-1}\bbX \to\,
                      ^{n}\bbX \to\, ^{n+1}\bbX \to \cdots$$
of objects in $\B$, and whose morphisms are {\bf strict} morphisms
of complexes, that is, morphisms $\bfX \to \bfY$ such that for every
$n$ the morphism $^{n}\bbX \to\, ^{n}\bbY$ is strict; see
$\S$\ref{SS:definition}.

We define two classes of morphisms $\sis \subset \qis$ in
$\Chst(\B)$. The class $\sis$ consists of morphisms $s \: \bfX \to
\bfY$ that become  isomorphisms in $\Ch(\B)$; {\em note that
$s^{-1}$ may no longer be strict, so $s$ is not necessarily an
isomorphism in $\Chst(\B)$}. The class $\qis$ consists of all
quasi-isomorphisms in $\Chst(\B)$.

The class  $\sis$ is indeed a localizing class. This follows from
Lemma \ref{L:roof2} below.

\begin{lem}{\label{L:roof1}}
  Let $P \: \bbX \to \bbY$ be a morphism in $\B$. Then,
  there is a functorial commutative diagram
       $$\xymatrix@=10pt@M=8pt@C=14pt{ & \bbE \ar [ld]_s \ar[rd]^g & \\
          \bbX \ar[rr]_P \ar[rd]_h && \bbY \ar[ld]^{t} \\
          & \bbF &}$$
  in $\B$ such that $s$, $t$,  $g$ and $h$  are strict
  ($\S$\ref{SS:definition}) and $s$ and $t$ are  isomorphisms
  (note that $s^{-1}$ and $t^{-1}$ are no longer strict).
\end{lem}

\begin{proof}
    First we prove the existence of $s$ and $g$.
    Consider the diagram for $P$
           $$\xymatrix@C=8pt@R=6pt@M=6pt{ X^{-1} \ar[rd]^{\kappa}
           \ar[dd]
                            & & Y^{-1} \ar[ld]_{\iota} \ar[dd]  \\
                  & E \ar[ld]^{\sigma} \ar[rd]_{\rho}  & \\
                  X^0 & & Y^0       }$$
   We define
       $$\bbE:=[X^{-1}\oplus Y^{-1}\llra{\kappa+\iota} E].$$
   The strict map $s\: \bbE \to \bbX$ is given by $(\pr_1, \sigma)$ and is
   easily seen to be an isomorphism. The map  $g \: \bbE \to \bbY$
   is defined by  $(\pr_2,\rho)$.

   The construction of $t$ and $h$ is similar. We take
     $$\bbF:=[E \llra{(\sigma,\rho)} X^0\oplus Y^0].$$
   The strict map $t\: \bbY \to \bbF$ is given by $(\iota, \inc_2)$ and is
   easily seen to be an isomorphism. The map $h \: \bbX \to \bbF$ is
   defined by $(\kappa,\inc_1)$.

  Let us prove the functoriality of $\bbE$. Consider the commutative diagram
      $$\xymatrix@C=10pt@R=10pt@M=6pt{ \bbX' \ar[r]^{P'}
                   \ar[d]_u & \bbY' \ar[d]^v \\
                                     \bbX \ar[r]_{P} & \bbY }$$
  and let $\bbE$ and $\bbE'$ be constructed as above. Let
  $w \: \bbE \to \bbE'$
  be $s^{-1}\circ u\circ s'$. It is easy to see that $w$ commutes
  with both $s$ maps and the $g$ maps; in fact $w$ is uniquely
  determined by this property. (Observe that we did not require
  $w$ to be strict.)

   The functoriality of $\bbF$ is proved in a similar way.
\end{proof}

\begin{lem}{\label{L:roof2}}
  Let $f \: \bfX \to \bfY$ be a morphism in $\Ch(\B)$.
  Then, there is a  commutative diagram
       $$\xymatrix@=10pt@M=8pt@C=14pt{ & \bfE \ar [ld]_s \ar[rd]^g & \\
          \bfX \ar[rr]_f \ar[rd]_h && \bfY \ar[ld]^{t} \\
                                                     & \bfF &}$$
  in $\Ch(\B)$ such that $s$, $t$,  $g$ and $h$  are in $\Chst(\B)$
  and $s$ and $t$ are  in $\sis$.
\end{lem}

\begin{proof}
  This follows immediately from Lemma \ref{L:roof1}.
\end{proof}

\begin{prop}{\label{P:derived}}
  The inclusion $\Chst(\B) \hookrightarrow \Ch(\B)$ induces  the
  following  equivalences of categories:
     $$\sis^{-1}\Chst(\B) \risom \Ch(\B),$$
     $$\qis^{-1}\Chst(\B) \risom \mcD(\B).$$
\end{prop}

\begin{proof}
  The first equivalence follows immediately from Lemma \ref{L:roof2}.
  The second equivalence follows from the first equivalence.
\end{proof}

We will need the following  lemma in Section \ref{S:iterate}.

\begin{lem}{\label{L:localizing}}
  Let $a < b$ be two integers.  Then the full subcategory
  $^{[a,b]}\Chst(\B)$ of
  $\Chst(\B)$ consisting of complexes concentrated in degrees lying in
  the interval
  $[a,b]$ is localizing with respect to both $\sis$ and $\qis$.
  That is, we have fully faithful functors:
    $$P\: \sis^{-1}\,^{[a,b]}\Chst(\B) \to \sis^{-1}\Chst(\B),$$
    $$Q\: \qis^{-1}\,^{[a,b]}\Chst(\B) \to \qis^{-1}\Chst(\B).$$
  (Here, by abuse of notation, we have denoted $\sis^{-1}\cap
  \,^{[a,b]}\Chst(\B)$ and $\qis^{-1}\cap
  \,^{[a,b]}\Chst(\B)$ also by $\sis$ and $\qis$.)
  The same thing is true if we take the full subcategory of complexes
  $\bfX$ in $^{[a,b]}\Chst{\B}$ such  that
  $H^{a}(\bfX) \in \mcF'$ and $H^{b}(\bfX) \in \mcT'$.
\end{lem}

\begin{proof}
  The case of $\sis$ is obvious.
  Let us prove the case of $\qis$. Let $\sfK(\B)$ and $^{[a,b]}\sfK(\B)$
  be the  homotopy categories of
  $\Chst(\B)$ and $^{[a,b]}\Chst(\B)$. It is enough to prove
  the  statement for the full subcategory
  $^{[a,b]}\sfK(\B) \subseteq
  \sfK(\B)$. We will denote the class of quasi-isomorphisms in
  $\sfK(\B)$ by $\mcS$. Note that this is a localizing class.
  By abuse of notation, we denote
 $\mcS\cap
  \,^{[a,b]}\sfK(\B)$ also
  by $\mcS$.

   Let $\taudg{a},\taudl{b} \: \Chst(\B) \to \Chst(\B)$
   be the usual truncation functors that we know from the theory of
   $t$-structures, and let
   $\tau_{[a,b]}=\taudg{a}\circ\taudl{b} \: \Chst(\B) \to \,^{[a,b]}\Chst(\B)$.
   We  use the same notation $\tau_{[a,b]}$ for the induced
   functor on the homotopy categories as well as the localized
   categories.
    Let $\bfX,\bfY \in \, ^{[a,b]}\sfK(\B)$. We have to show that
    the map
    $$\alpha \: \Hom_{\mcS^{-1}(\,^{[a,b]}\sfK(\B))}(\bfX,\bfY)\to
                                 \Hom_{\mcS^{-1}\sfK(\B)}(Q\bfX,Q\bfY).$$
  induced by $Q$ is an isomorphism. Let
   $$\beta \:  \Hom_{\mcS^{-1}\sfK(\B)}(Q\bfX,Q\bfY) \to
                       \Hom_{\mcS^{-1}(\,^{[a,b]}\sfK(\B))}(\bfX,\bfY)$$
  be the map induced by $\tau_{[a,b]}$. We show that $\alpha$ and
  $\beta$ are inverse to each other. It is clear that
  $\beta\circ\alpha=1$. To prove  $\alpha\circ\beta=1$,
  let $\tilde{f}\in\Hom_{\mcS^{-1}\sfK(\B)}(Q\bfX,Q\bfY)$
  be given by the roof
     $$\xymatrix@C=10pt@R=8pt@M=6pt{& \bfZ\ar[ld]_{s} \ar[rd]^{f} & \\
           \bfX && \bfY }$$
  where $s\in \mcS$. Then $\alpha\circ\beta(\tilde{f})$ is
  given by the roof
    $$\xymatrix@C=14pt@R=12pt@M=6pt{
                 & \tau_{[a,b]}\bfZ\ar[ld]_(0.55){\tau_{[a,b]}(s)}
                          \ar[rd]^(0.55){\tau_{[a,b]}(f)} & \\
                                                    \bfX && \bfY }$$
  To show that this is equal to $\tilde{f}$ we construct a
  commutative diagram
   $$\xymatrix@C=16pt@R=14pt@M=6pt{
            && \mathbf{V} \ar[ld]_r \ar[rd]^h && \\
          & \bfZ \ar[ld]_s \ar[rrrd]_f &&
               \tau_{[a,b]}\bfZ\ar[llld]^{\tau_{[a,b]}(s)}
                                \ar[rd]^{\tau_{[a,b]}(f)} & \\
                       \bfX &&&& \bfY
   }$$
  in which $s\circ r \in \mcS$. This diagram is easy to construct.
  Simply take $\mathbf{V}=\taudl{b}\bfZ$ and let $r$ and $h$ be the
  unit and the counit of the adjunction for the functors $\taudl{b}$
  and $\taudg{a}$, respectively.

  The proof in the case of a torsion pair is exactly the same, once
  we take $\taudg{a}$ and $\taudl{b}$ to be the truncation functors
  of the  $t$-structure corresponding to the torsion pair.
\end{proof}
\section{Decorated complexes in $\A$}{\label{S:decorated}}

Let $\A$ be an abelian category. We will not fix a torsion pair on
$\A$ yet.

A {\bf decorated complex} in $\A$ consists of the following data:

\begin{itemize}

\item[$\diamond$] A chain complex in $\A$
     $$\Eb: \ \ \cdots \torel{\delta}  E^{n-1} \torel{\delta} E^{n}
      \torel{\delta} E^{n+1} \torel{\delta} \cdots;$$

\item[$\diamond$]  A graded subobject $\Mb$ of the underlying graded
 object of $\Eb$. That is, a
 sequence $M^{n} \subseteq E^n$, $n \in \bbZ$, of subobjects,
 not necessarily respected by $\delta$.
\end{itemize}

A morphism $(\Eb,\Mb) \to (\Fb,\Nb)$ of decorated complexes is a map
$f \: \Eb \to \Fb$ of complexes such that for every $n$,
$f(M^n)\subseteq N^n$.

We denote  the category of decorated complexes
 by $\Dec(\A)$.

\subsection{Zero and full decorations}{\label{SS:zerofull}}

 There are two natural decorations on every complex $\Eb$:
the {\em zero} decoration, in which all $M^n$ are zero, and the {\em
full} decoration, in which $M^n=E^n$, for all $n$. We have the
corresponding fully faithful embeddings
  $$i\: \Ch(\A) \hookrightarrow  \Dec(\A), \ \text{zero decoration},$$
  $$j\: \Ch(\A) \hookrightarrow  \Dec(\A), \ \text{full decoration}.$$

 Both of these functors have both left and right adjoints. The left
 adjoint to $j$ is the forgetful functor (forgetting the
 decoration), and
 the right adjoint  is what is denoted by $\mcH^{-1}$ in
 $\S$\ref{SS:deccoh}. The left adjoint to $i$ is
  $\mcH^{0}[-1]$, and the right adjoint to $i$ is the forgetful
 functor.

\subsection{Various cohomologies associated to decorated complexes}
 {\label{SS:deccoh}}

Other than the usual cohomology of the complex $\Eb$, to a decorated
complex $(\Eb,\Mb)$ we can associate the following two families of
cohomologies:
  $$H^{-1,n}(\Eb,\Mb):=\ker(M^n \torel{\delta}
    E^{n+1}/M^{n+1}),$$
  $$H^{0,n}(\Eb,\Mb):=\coker(M^n \torel{\delta}
    E^{n+1}/M^{n+1}).$$
These cohomologies themselves fit into two complexes:
{\footnotesize
$$\mcH^{-1}(\Eb,\Mb): \ \ \ \cdots \to H^{-1,n-1}(\Eb,\Mb) \to
   H^{-1,n}(\Eb,\Mb) \to
   H^{-1,n+1}(\Eb,\Mb) \cdots,$$
$$\mcH^{0}(\Eb,\Mb): \ \ \ \cdots \to  H^{0,n-1}(\Eb,\Mb) \to
   H^{0,n}(\Eb,\Mb) \to
   H^{0,n+1}(\Eb,\Mb)\to  \cdots.$$}
Let us denote the cohomologies of these complexes by
$H^n(\mcH^{0}(\Eb,\Mb))$ and $H^{n}(\mcH^{-1}(\Eb,\Mb))$, or
$H^n(\mcH^{-1})$ and $H^n(\mcH^{0})$, if $(\Eb,\Mb)$ is understood
from the context.

\begin{prop}{\label{P:cohofcoh}}
  There are natural morphisms
   $H^{n-1}(\mcH^{0})\llra{\partial}H^{n+1}(\mcH^{-1})$ fitting in
   a long exact sequence
   $$\cdots \to H^{n-2}(\mcH^{0})  \torel{\partial}
      H^{n}(\mcH^{-1})
        \to H^n(E) \to H^{n-1}(\mcH^{0}) \torel{\partial}
         \qquad \qquad\qquad \qquad \hbox{}$$ $$ \hbox{} \qquad \qquad
                                   \qquad   H^{n+1}(\mcH^{-1})
      \to H^{n+1}(E) \to H^{n}(\mcH^{0}) \torel{\partial}
      H^{n+2}(\mcH^{-1})\to \cdots.$$

\end{prop}
\begin{proof}
   For simplicity, we denote $\Eb$ by $E$. We put a filtration on
   the complex $E$ by setting
       $$F^nE=M^n\oplus\bigoplus_{n<i}E^i.$$
   This gives rise to a spectral sequence whose second page is exactly
   the union of the complexes $\mcH^{-1}$ and $\mcH^{0}$. The
   differentials of the third page are  the morphisms
   $H^{n-1}(\mcH^{0})\llra{\partial}H^{n+1}(\mcH^{-1})$, and after the
   third page all the differential become zero for degree reasons.
   The fact that this spectral sequence converges to $H^n(E)$
   is equivalent to the existence of the above long exact sequence.

   An  alternative proof can be obtained by showing that the natural
   map of complexes $E/\mcH^{-1}\to\mcH^0[-1]$ is a
   quasi-isomorphism.
\end{proof}

 We define two classes of morphisms $\sis$ and $\qis$ in
$\Dec(\A)$. The former is the class of all morphisms in $\Dec(\A)$
which induce isomorphisms on all $H^{i,n}$,  $i=-1,0$, $n \in \bbZ$.
The latter is the class of all quasi-isomorphisms, that is, all
morphisms in $\Dec(\A)$ which induce isomorphisms on the usual
cohomologies $H^n$.

\begin{prop}{\label{P:sis}}
  We have $\sis \subset \qis$.
\end{prop}

\begin{proof}
Follows from Proposition \ref{P:cohofcoh}.
\end{proof}

In Section \ref{S:complexes} we will see the relation between the
classes $\sis$ and $\qis$ in $\Chst(\B)$, as defined in Section
\ref{S:complexesstrict}, and the classes $\sis$ and $\qis$ defined
above.
  Presumably $\sis$ is not a localizing class in $\Dec(\A)$.
  However, its restriction to the full subcategory $\Ch(\A,\mcT,\mcF)$
  defined  in $\S$\ref{SS:dectorsion} is a localizing class.

\subsection{The cohomological functor $\bbH \: \Ch(\A) \to \B$
associated to a torsion pair}{\label{SS:cohfunc}}

 Assume now that $\A$ is equipped with a
torsion pair $(\mcT,\mcF)$. There is a cohomological functor $\bbH
\: \Ch(\A) \to \B$ obtained from the $t$-structure associated to the
torsion pair $(\mcT,\mcF)$. We give an explicit description of this
cohomological functor.

 Let $\Eb$ be a complex in $\Ch(\A)$. We
define $A^n$, $n\in \bbZ$, to be the subobject of $E^n$ such that
$$\im(E^{n-1} \torel{\delta} E^{n}) \subseteq
          A^n \subseteq \ker(E^n \torel{\delta} E^{n+1}),$$
and
        $$A^n /\im(E^{n-1} \torel{\delta} E^{n}) \in \mcT, \ \
          \text{and}  \
              \ker(E^n \torel{\delta} E^{n+1})/A^n \in \mcF.$$
These properties uniquely determine $A^n$.

 For a chain complex
$\Eb$ in $\Ch(\A)$ we define
  $$\bbH^n(\Eb):=[E^n/A^n \llra{\delta} A^{n+1}] \in \B.$$

 \begin{prop}{\label{P:cohomology}}
    Let $\Eb$ be a chain complex in $\A$. Then, for every $n$, there
    is a short exact sequence
       $$ 0 \to H^0(\bbH^{n-1}(\Eb)) \to H^n(\Eb) \to
       H^{-1}(\bbH^n(\Eb)) \to 0.$$
   Note that $H^0(\bbH^{n-1}(\Eb)) \in \mcT$ and $H^{-1}(\bbH^n(\Eb))
   \in \mcF$.
 \end{prop}

\begin{proof} This follows immediately from the definition of
  $\bbH^n(\Eb)$.
\end{proof}

To have a better grasp of the cohomological functor $\bbH$, it is
perhaps useful to consider the two extreme cases in which
$\mcT=\{0\}$ or $\mcF=\{0\}$. In both cases, we have a  natural
identification $\B=\A$. In the first case, $\bbH^n=H^n$, and in the
second case $\bbH^n=H^{n+1}$. For a general  torsion pair, $\bbH^n$
has a little bit of $H^n$ and a little bit of $H^{n+1}$, as we saw
in Proposition \ref{P:cohomology}.

The next corollary follows immediately from Proposition
\ref{P:cohomology}.

\begin{cor}{\label{C:qis}}
  A morphism $\Eb \to \Fb$ in $\Ch(\A)$
  is a quasi-isomorphism if and only if it induces
  isomorphisms on all $\bbH^n$.
\end{cor}

\subsection{Decorations in the presence of a torsion
pair}{\label{SS:dectorsion}

 Let $(\Eb,\Mb)$ be a decorated complex. We say that the decoration
is {\bf compatible} with the torsion pair $(\mcT,\mcF)$ if the
following condition is satisfied:

\begin{itemize}
  \item[$\blacktriangleright$] For every $n$, $H^{-1,n}(\Eb,\Mb) \in \mcF$
    and $H^{0,n}(\Eb,\Mb) \in \mcT$.
\end{itemize}

The decorated complexes whose decoration is compatible with the
torsion pair form a full subcategory of $\Dec(\A)$ which we denote
by $\Ch(\A,\mcT,\mcF)$. For an object $(\Eb,\Mb)$ in
$\Ch(\A,\mcT,\mcF)$ we define $\bbH^n(\Eb,\Mb)=\bbH^n(\Eb)$.

If we intersect the classes $\sis, \qis \subset \Dec(\A)$ defined in
$\S$\ref{SS:deccoh} with the subcategory $\Ch(\A,\mcT,\mcF)$, we
obtain two classes of morphisms in $\Ch(\A,\mcT,\mcF)$, for which we
use the same notation. The class $\sis$ is a localizing class. This
follows, for example, from Theorem \ref{T:main} below.

We will see in Section \ref{S:complexes}  that there exists a
natural equivalence of categories $\Chst(\B) \risom
\Ch(\A,\mcT,\mcF)$ under which the classes $\sis, \qis \subset
\Chst(\B)$ defined in Section \ref{S:complexesstrict} exactly
correspond to the classes $\sis, \qis \subset \Ch(\A,\mcT,\mcF)$
defined above.

\subsection{Homological algebra in $\Dec(\A)$}{\label{SS:homological}
  The category $\Dec(\A)$ should be thought of as a generalization
  of the category $\Ch(\A)$ in the sense that we can do homological
  algebra in $\Dec(\A)$. This means, the usual notions of
  homological algebra (such as, {\em mapping cylinder} of a
  morphism, {\em mapping cone} of a morphism, {\em chain homotopy}
  between morphisms, and so on) can be defined in $\Dec(\A)$.

 Let us show how mapping cylinders are defined in $\Dec(\A)$.
 (Essentially, all other definitions  can be formally reduced to
 this one.) Let $f \: (\Eb,\Mb) \to  (\Fb,\Nb)$ be a morphism in
 $\Dec(\A)$ and denote $\Eb \to \Fb$ by $\overline{f}$. We define
 $\Cyl(f)$ to be the usual mapping cylinder
 $\Cyl(\overline{f})=\Eb\oplus E^{\bullet+1}\oplus \Fb$, endowed
 with the direct sum of the decorations of its components. We have
 natural morphisms $(\Eb,\Mb) \hookrightarrow \Cyl(f)$ and
 $(\Fb,\Nb) \hookrightarrow \Cyl(f)$ in $\Dec(\A)$. In the case
 where $f$ is the identity, $\Cyl(f)$ is the cylinder of
 $(\Eb,\Mb)$, which can be used to define decorated chain
 homotopies. The quotient $\Cone(f):=\Cyl(f)/(\Eb,\Mb) \in \Dec(\A)$
 is the mapping cone of $f$, and so on.

\begin{lem}{\label{L:cone}}
 Let $f \:  (\Eb,\Mb) \to  (\Fb,\Nb)$ be a morphism
  in $\Dec(\A)$, and let $\Cone(f)$ be its decorated cone
  as defined above. Then, we have the  isomorphisms
    $$H^{i,n}(\Cone(f))\cong H^{i,n+1}(\Eb,\Mb)\oplus
    H^{i,n}(\Fb,\Nb), \ \ \ i=-1,0.$$
\end{lem}

\begin{proof}
 Straightforward.
\end{proof}

\begin{rem}
 Decorated homotopies do not necessarily induce isomorphisms on
 $H^{i,n}$, but they induce chain homotopies on the level of complexes
 $\mcH^{-1}$ and $\mcH^0$.
\end{rem}

   Passing to decorated chain homotopy classes of morphisms in
   $\Dec(\A)$ we obtain the {\bf homotopy category} $\KDec(\A)$ of
   decorated chain complexes. This is a triangulated category with
   the usual shift functor. Doing the same with $\Ch(\A,\mcT,\mcF)$ we
   obtain a full triangulated subcategory of $\KDec(\A)$ which we
   denote by $\sfK(\A,\mcT,\mcF)$. (Here, we have used the fact that
   the cone of a morphism in $\Ch(\A,\mcT,\mcF)$ is again in
   $\Ch(\A,\mcT,\mcF)$; see Lemma \ref{L:cone} above.)

   The class $\qis \subseteq \KDec(\A)$ is a localizing class, and so is its
   intersection with $\sfK(\A,\mcT,\mcF)$, for which we use the
   same notation $\qis$. This is because both are defined
   by a cohomological functor. More precisely, we are using the
   following.

   \begin{lem}[\oldcite{Weibel}, Proposition 10.4.1]{\label{L:Weibel}}
      Let $\sfK$ be a triangulated category, and let $\mcS$ be
      the class of quasi-isomorphisms with respect to a certain
      cohomological functor. Then $\mcS$ is a localizing class.
   \end{lem}

  As in classical homological algebra, the fact that $\qis$ becomes a
  localizing class is very useful. Note that
  $$\qis^{-1}\KDec(\A)\cong \qis^{-1}\Dec(\A) \ \ \ \text{and} \ \ \ \
  \qis^{-1}\sfK(\A,\mcT,\mcF)\cong \qis^{-1}\Ch(\A,\mcT,\mcF).$$

  There are two other triangulated subcategories of $\sfK(\A,\mcT,\mcF)$
  that are less important for us (they will only be used in the proof
  of Theorem \ref{T:HRS}). We will end this section by giving their
  definitions. The first category, denoted $\sfF(\A,\mcT,\mcF)$,
  is the full subcategory of $\sfK(\A,\mcT,\mcF)$ consisting of
  decorated complexes $(\Eb,\Mb)$ such that $\Eb$ is a complex of
  free objects, i.e., $E^n \in \mcF$ for all $n$. This is a
  triangulated category because the cone of morphism between two
  complexes of free objects is again a complex of free objects.

  The category $\sfF(\A,\mcT,\mcF)$ contains a subcategory
  $\sfFF(\A,\mcT,\mcF)$ consisting of the complexes with the full
  decoration. We will need the following lemmas for the proof of
  Theorem \ref{T:HRS}.

  \begin{lem}{\label{L:equivalence}}
   Let $\sfK$ be a triangulated category, and let $\mcS$ a
   localizing class in $\sfK$ that is defined by a cohomological
   functor. Let $\sfF$ be a full triangulated subcategory of $\sfK$,
   and set $\mcS_{\sfF}=\mcS\cap\sfF$. Assume either
   of the following holds:
  \begin{itemize}
   \item[$\triangleright$] For every $X \in \sfK$,
      there exists a quasi-isomorphism $F \to X$ with $F \in \sfF$;
   \item[$\triangleright$] For every $X \in \sfK$,
      there exists a quasi-isomorphism $X \to F$ with $F \in \sfF$.
  \end{itemize}
   Then the functor $\mcS_{\sfF}^{-1}\sfF \to
   \mcS^{-1}\sfK$ is an equivalence of triangulated
   categories.
  \end{lem}

\begin{proof}
   Essential surjectivity is obvious, so it is enough to show that
   $\sfF$ is a localizing subcategory of $\sfK$. This follows from
  \cite{GeMa}, Proposition III.2.10 (page 151). (Note that, by Lemma
  \ref{L:Weibel}, $\mcS_{\sfF}$ is automatically a
   localizing class.)
\end{proof}

  \begin{lem}{\label{L:full}}
      The inclusion $\sfFF(\A,\mcT,\mcF) \subset \sfF(\A,\mcT,\mcF)$
      induces an equivalence of triangulated categories
        $$\qis^{-1}\sfFF(\A,\mcT,\mcF) \cong
                         \qis^{-1}\sfF(\A,\mcT,\mcF).$$
  \end{lem}

  \begin{proof}
    For   every $(\Eb,\Mb)$
     in $\sfF(\A,\mcT,\mcF)$ we have a quasi-isomorphism
     $(\Eb,\Mb) \to (\Eb,\Eb)$, where $(\Eb,\Eb) \in
     \sfFF(\A,\mcT,\mcF)$. The result follows from the second
     case of Lemma \ref{L:equivalence}.
  \end{proof}

  Dually, there are  triangulated subcategories
     $$\sfZT(\A,\mcT,\mcF)
       \subset \sfT(\A,\mcT,\mcF) \subset \sfK(\A,\mcT,\mcF),$$
   where $\sfT(\A,\mcT,\mcF)$ consists of
   decorated complexes of torsion objects and
  $\sfZT(\A,\mcT,\mcF)$ is its full subcategory consisting of
  complexes with zero decoration. The above discussion applies
  to these categories as well.

\begin{rem}{\label{L:bounded}}
  The discussion of this subsection applies to the case where
  the complexes are bounded (above, below, or both).
\end{rem}

\section{Complexes in $\B$ and the derived category $\mcD(\B)$}
{\label{S:complexes}}

In this section we give an alternative description of $\Chst(\B)$;
see Theorem \ref{T:main}. This, together with Proposition
\ref{P:derived} enables us to give a simple description for the
derived category $\mcD(\B)$.

\subsection{Description of $\Ch(\B)$ via decorated complexes}
{\label{SS:main}}

\begin{thm}{\label{T:main}}
  There is an equivalence of categories
  $$\Tot \: \Chst(\B) \risom \Ch(\A,\mcT,\mcF).$$
  Under this equivalence, the images of
  $\sis, \qis \subset \Chst(\B)$
  are exactly $\sis, \qis \subset \Ch(\A,\mcT,\mcF)$
\end{thm}

Recall that $\qis \subset \Ch(\A,\mcT,\mcF)$ is the class of all
quasi-isomorphisms, that is, all morphisms $(\Eb,\Mb) \to (\Fb,\Nb)$
such that $\Eb \to \Fb$ is a quasi-isomorphism in $\Ch(\A)$. The
class $\sis \subset \qis$ consists of those morphisms $(\Eb,\Mb) \to
(\Fb,\Nb)$ that induce  isomorphisms on $\mcH^{-1}$ and $\mcH^0$;
see $\S$\ref{SS:deccoh}.

The following corollary is immediate from Theorem \ref{T:main}.

\begin{cor}{\label{C:main}}
   The functor $\Tot$ induces an equivalence of categories:
   $$\Ch(\B) \risom \sis^{-1} \Ch(\A,\mcT,\mcF),$$
   $$\mcD(\B) \risom \qis^{-1} \Ch(\A,\mcT,\mcF).$$
\end{cor}

\begin{proof}
  Follows from Proposition \ref{P:derived}.
\end{proof}

We prove Theorem \ref{T:main} by giving a step by step of
simplification of what goes into the definition of a chain complex
in $\B$. We begin by complexes of length two.

\subsection{Complexes of length 2 in $\B$}{\label{SS:complexes1}}

Consider the morphisms $P \: \bbX \to \bbY$ and $Q\:\bbY \to \bbZ$
as in the following diagram:
  $$ \xymatrix@C=8pt@R=6pt@M=6pt{ X^{-1} \ar[rd]  \ar[dd]
             & & Y^{-1} \ar[ld]_{\iota}  \ar[rd]^{\kappa'}  \ar[dd]
             & & Z^{-1} \ar[ld]  \ar[dd]   \\
     & E \ar[ld]  \ar[rd]_{\rho}  & & F \ar[ld]^{\sigma'}  \ar[rd]  & \\
                                     X^0 & & Y^0   & & Z^0    }$$
Let $A \subseteq E$ be defined as in $\S$\ref{SS:kernel}, and let $B
\subseteq F$ be the corresponding object for the morphism $Q$.

\begin{lem}{\label{L:composition}}
The composition $Q\circ P$ is zero if and only if there is a
morphism $f \: E/A \to B$ in $\A$  making the following diagram
commute:
   $$\xymatrix@C=-6pt@R=9pt@M=6pt{
         &&&&& Y^{-1} \ar[rrrrrd]^{\kappa'} \ar[llllld]_{\iota} &&&&& \\
    E \ar[rrrr] \ar[rrrrrd]_{\rho} &&&& E/A   \ar @{..>} [rr]^{\exists f}
                            && B \ar[rrrr] &&&& F \ar[llllld]^{\sigma'}\\
                                     &&&&& Y^0 &&&&&             }$$
In this case, the monomorphism $\im P \to \ker Q$ is realized by the
strict morphism
  $$\xymatrix@C=8pt@R=6pt@M=6pt{ Y^{-1}  \ar[rr]^{\id}  \ar[dd]_{\iota}
                            & & Y^{-1}   \ar[dd]^{\kappa'} \\
                  &          & \\
                  E/A \ar[rr]_{f} & & B       }$$
\end{lem}

\begin{proof}
  Recall from $\S$\ref{SS:kernel} and $\S$\ref{SS:epimono} that
  $\im P=[Y^{-1}\to E/A]$
  and $\ker Q=[Y^{-1}\to B]$. The composition $Q\circ P$ being zero
  is equivalent to the existence of a commutative triangle
  $$\xymatrix@C=0pt@R=13pt@M=6pt{\im P \ar[rr]^{\varphi}
         \ar[rd] & & \ker Q \ar[ld] \\
                              & \bbY & }$$
  If we unravel this triangle, we see that it is  equivalent to
  the diagram required in the lemma.
\end{proof}

\begin{cor}{\label{C:cohomology}}
  In the sequence $\bbX \llra{P} \bbY \llra{Q} \bbZ$  above the
  cohomology at $\bbY$ is given by
    $$\bbH:= [E/A \llra{f} B].$$
\end{cor}

\begin{cor}{\label{C:exact}}
  A sequence $\bbX \llra{P} \bbY \llra{Q} \bbZ$ as above is exact
  at $\bbY$ if and only if there is an isomorphism $E/A \risom B$
  respecting  the  morphisms $\iota$, $\rho$, $\kappa'$, and $\sigma'$.
\end{cor}

\subsection{An alternative way of looking at complexes in $\B$}
{\label{SS:complexes2}}

Let $P$ and $Q$ be as in $\S$\ref{SS:complexes1}. By a {\em link}
 from $P$ to $Q$ we mean a morphism
$\delta \: E \to F$ such that:

\begin{itemize}

  \item[$\mathbf{L}$)] The following diagram commutes and the
  horizontal  sequence  is a complex:
    $$\xymatrix@C=2pt@R=4pt@M=6pt{ &&&&& Y^{-1} \ar[rd]^{\kappa'}
                                                  \ar[ld]_{\iota} &&&& \\
        0 \ar[rr] && X^{-1} \ar[rr] && E \ar[rd]_{\rho} \ar[rr]^{\delta} &&
                       F \ar[ld]^{\sigma'} \ar[rr] && Z^0 \ar[rr] && 0  \\
                             &&&&& Y^0 &&&&& }$$
\end{itemize}

\begin{prop}{\label{P:link}}
    Consider the  sequence $\bbX \llra{P} \bbY \llra{Q} \bbZ$.
    Then $Q\circ P=0$ if and only if there exists a link from $P$
    to $Q$. If such a link exists then it is unique.
\end{prop}

\begin{proof}
    One implication is trivial from Lemma \ref{L:composition}.
    To prove the reverse implication,
    let $A \subseteq E$ and $B \subseteq F$ be as in
   Lemma \ref{L:composition}. We need to show that
   every  $\delta \: E \to F$ as in ($\mathbf{L}$) necessarily
   vanishes on $A$ and factors through $B$; the result will then
   follow from Lemma \ref{L:composition}.

   Let us prove that $\delta$ vanishes on $A$. Since
   $\delta(X^{-1})=0$,
   we have an induced map $\overline{\delta} \: H^{-1}(C(P)) \to F$.
   By definition, the image $T$ of $A$ in $H^{-1}(C(P))$ is the
   torsion   part of $H^{-1}(C(P))$. Observe  that
   $\overline{\delta} \: H^{-1}(C(P)) \to F$ factors through the
   kernels of
   both $\sigma' \: F\to Y^0$ and $\rho' \: F \to Z^0$, and that
   $\ker\sigma'\cap\ker\rho'=\ker (d \: Z^{-1} \to Z^0)$ belongs to
   $\mcF$. So $\overline{\delta}(T)=0$. That is $\delta(A)=0$.

   The proof that $\delta$ factors through $B$ is similar.

   We now prove the uniqueness. Let $\delta$ and
    $\delta'$ be two links, and set $\epsilon=\delta-\delta'$.
    By the commutativity condition of ($\mathbf{L}$), $\epsilon$
    vanishes on $Y^{-1}$. Since the horizontal sequence is a complex,
    $\epsilon$ vanishes on $X^{-1}$
    as well. Hence, $\epsilon$ factors through
    $E/(\kappa(X^{-1})+\iota(Y^{-1})\cong H^0(\bbX)$.

    Similarly, by the commutativity condition of ($\mathbf{L}$),
    $\sigma'\circ\epsilon=0$, and since the  horizontal sequence is
    a complex, $\rho'\circ\epsilon=0$, where $\rho'$ is the morphism
   $\rho' \: F \to Z^0$. This implies that
    $\epsilon$ factors through $\ker\rho' \cap \ker \sigma'\cong
    H^{-1}(\bbZ)$.

     Putting these together, we see that $\epsilon$ factors through
     a map $H^0(\bbX) \to H^{-1}(\bbZ)$. But such a map is
     necessarily zero because $H^0(\bbX)$ is in $\mcT$ and
     $H^{-1}(\bbZ)$ is in $\mcF$. Therefore, $\epsilon$ is the zero
     map. That is, $\delta=\delta'$.
\end{proof}

\begin{prop}
    The complex
        $$0\llra{}\bbX \llra{P} \bbY \llra{Q} \bbZ\llra{}0$$
    is short exact if and only of the middle sequence in $(\mathbf{L})$
    is exact, $H^{-1}(C(P)) \in \mcF$, and $H^{-1}(C(Q)) \in \mcT$.
\end{prop}

\begin{proof}
  Follows from Corollary \ref{C:monoepi} and Corollary \ref{C:cohomology}.
\end{proof}

Using the idea of  link, we see that a chain complex
 $$\bfX= \ \ \cdots \to\, ^{n-2}\bbX \to\, ^{n-1}\bbX \to\,
    ^{n}\bbX \to\, ^{n+1}\bbX \to\, ^{n+2}\bbX \to \cdots,$$
with $^n\bbX= [d: \,^nX^{-1} \to\, ^nX^0]$,  is equivalently
described by the diagram
 {\small
  $$ \xymatrix@C=-3pt@R=6pt@M=6pt{& ^{n-2}X^{-1} \ar[ld]_{\iota}
    && ^{n-1}X^{-1} \ar[ld]_{\iota}         && ^nX^{-1} \ar[ld]_{\iota}
    && ^{n+1}X^{-1}   \ar[ld]_{\iota}       && ^{n+2}X^{-1} \ar[ld]_{\iota}  &\\
       \cdots  \ar[rr]^{\delta} && ^{n-1}E \ar[ld]^{\sigma} \ar[rr]^{\delta}
    && ^{n}E \ar[ld]^{\sigma}  \ar[rr]^{\delta}
    && ^{n+1}E \ar[ld]^{\sigma} \ar[rr]^{\delta}
    &&  ^{n+2}E \ar[ld]^{\sigma} \ar[rr]^{\delta} && \cdots. \ar[ld]^{\sigma} \\
        &   ^{n-2}X^0 & & ^{n-1}X^0 & & ^nX^0   & & ^{n+1}X^0   &&^{n+2}X^0 & }$$
     }
Literally translating  all the commutativity and exactness
conditions that are needed to be satisfied, we arrive at the
following list of requirements:
\begin{itemize}

\item[$\mathbf{1}$)] \ Every NW-SE sequence is short exact;

\item[$\mathbf{2}$)] \ $\sigma\circ\delta\circ\iota=d$;

\item[$\mathbf{3}$)]  \ $\delta^2\circ\iota=0$ and $\sigma\circ\delta^2=0$;

\end{itemize}

The third axiom can be improved though.

\begin{lem}{\label{L:simplification}}
 We have $\delta^2=0$.
\end{lem}

\begin{proof}
 Let $^nC:=\ker(\delta\circ\iota)/\im(\sigma\circ\delta)$, i.e., the
 middle cohomology of the sequence $$^{n-1}X^{-1}
 \llra{\delta\circ\iota}\, ^{n}E \llra{\sigma\circ\delta}\,
 ^{n}X^0,$$ that is a complex by ($\mathbf{3}$). Let $^nA:=q^{-1}(T)
 \subseteq\, ^nE$, where $T \subseteq \, ^nC$ is the torsion part of
 $^nC$. As we saw in the proof of Proposition \ref{P:link}, $\delta
 \:\,^{n-1}E\to \,^{n}E $ factors through $^nA$, and $\delta
 \:^{n}E\to \,^{n+1}E$ vanishes on $^nA$. This means
 $$\delta(\,^{n-1}E) \ \subseteq  \ ^nA \  \subseteq \ \ker(\,^nE
 \torel{\delta}\,^{n+1}E).$$
   Therefore, $\delta^2=0$.
\end{proof}

Summarizing the above discussion, to give a complex $\bfX$ in $\B$
is equivalent to giving a diagram as above such that ($\mathbf{1}$)
and  ($\mathbf{2}$) are satisfied. Conditions ($\mathbf{1}$) and
($\mathbf{2}$) are saying that the objects $^nX^0$ are redundant and
can be deduced from the rest of the data. So we can scrape off all
the $^nX^0$, hence also the axioms ($\mathbf{1}$) and
($\mathbf{2}$), without losing any information about $\bfX$. This
leads to the definition of the functor $\Tot\: \Chst(\B) \to
\Ch(\A,\mcT,\mcF)$ that is discussed in the nest subsection.

\subsection{Definition of the functor $\Tot$ }{\label{SS:morphisms}}
We define the functor
    $$\Tot\: \Chst(\B) \to \Ch(\A,\mcT,\mcF)$$
as follows. Let $\bfX$ be a complex in $\B$ as in
$\S$\ref{SS:complexes2}. We set $\Tot(\bfX)=(\Eb,\Mb)$, where
$E^n:=\, ^nE$ and $M^n:=\iota(\,^{n}X^{-1})$.

To define the effect of $\Tot$ on morphisms, let $\bfX$ and  $\bfY$
be complexes in $\B$, and let $\Tot(\bfX)=(\Eb,\Mb)$ and
$\Tot(\bfY)=(\Fb,\Nb)$. If we use $\Tot$ and literally translate
what goes into the definition of a morphism of complexes $\bfX \to
\bfY$ in $\Chst(\B)$, we see that such a morphism is given by a
collection of morphisms $f_n \: E^n \to F^n$ in $\A$ satisfying the
following conditions:
\begin{itemize}
\item[$\mathbf{i}$)] For every $n$, $f_n(M^n) \subseteq N^n$;

\item[$\mathbf{ii}$)] For every $n$, the ``commutator''
     $$\epsilon_n:=f_{n+1}\circ\delta-\delta\circ f_n \: E^n \to
             F^{n+1}$$ 
     vanishes on $M^n$ and factors through $N^{n+1}$.
\end{itemize}

There are two problems here, however. The first problem is that,
 a priori, this is not quite
the same thing as a morphism in $\Ch(\A,\mcT,\mcF)$; we need that
$\epsilon_n$ be actually equal to zero. This is shown to be the case
in  Lemma \ref{L:zero} below. The second problem is that, a morphism
$\bfX \to \bfY$ does not, a priori, uniquely determine the
collection $\{f_n\}$. It only uniquely determines the effect of
$f_n$ on $M^n$ and $E^n/M^n$, but not on $E^n$. This is taken care
of in Lemma \ref{L:unique} below.

The idea of the proof of both lemmas is very similar to the proof of
Proposition \ref{P:link}.

\begin{lem}{\label{L:zero}}
   For every $n$, we have $\epsilon_n=0$. That is, the diagrams
     $$\xymatrix@=10pt@M=8pt@C=14pt{ E^n \ar[r]^{\delta} \ar[d]_{f_{n}}
           & E^{n+1}  \ar[d]^{f_{n+1}} \\
       F^n \ar[r]_{\delta}  & F^{n+1} }$$
    commute.
\end{lem}

\begin{proof}
    Note that, by definition, $E^n:=\, ^nE$ and
    $M^n:=\iota(\,^{n}X^{-1})$, $F^n:=\, ^nF$ and
    $N^n:=\iota(\,^{n}X^{-1})$.

    Condition ($\mathbf{ii}$) implies that the diagram commutes when
    restricted to $M^n \subseteq E^n$. It also implies that the
    diagram commutes when composed with the quotient map
    $q_{n+1} \: F^{n+1} \to F^{n+1}/N^{n+1}$.

   If we use ($\mathbf{i}$) and ($\mathbf{ii}$) with $n-1$, it
   follows that the diagram commutes when restricted to
   $\delta(M^{n-1})$ as well. That is, $\epsilon_n$ vanishes on
   $\delta(M^{n-1})$. Therefore, $\epsilon_n$ induces a morphism
   $$\varphi \: \coker(M^{n-1} \llra{p_{n}\circ \delta} E^{n}/M^{n})
   \to N^{n+1},$$
   where $p_{n} \: E^{n} \to E^{n}/M^{n}$ is the quotient
   map.

   Similarly, if we use ($\mathbf{i}$) and ($\mathbf{ii}$)
   with $n+1$, it follows that the diagram commutes when composed
   with $q_{n+2}\circ \delta \: F^{n+1} \to F^{n+2}/N^{n+2}$.
   This implies that $\epsilon_n$, hence also $\varphi$, factors
   through the kernel of $N^{n+1} \to F^{n+2}/N^{n+2}$.
   In other words, we obtain a morphism
      $$\varphi \: \coker(M^{n-1} \llra{p_{n}\circ \delta} E^{n}/M^{n})
        \to \ker (N^{n+1} \llra{q_{n+2}\circ \delta} F^{n+2}/N^{n+2}).$$
   Now observe that the left hand side is equal to
   $H^0(\,^{n-1}\bfX) \in \mcT$, and the right hand side is equal to
   $H^{-1}(\,^{n+1}\bfX) \in \mcF$. By the definition of a torsion
   pair, $\varphi$ has to be the zero map.
   This implies that $\epsilon_n=0$.
\end{proof}

\begin{lem}{\label{L:unique}}
  Let $\{f_n\}$ and $\{f'_n\}$ be two families of morphisms as above
  such that:
    \begin{itemize}
        \item[$\triangleright$] for every $n$, $f_n=f'_n \: M^n \to N^n$;

        \item[$\triangleright$]  for every $n$,
                              $f_n=f'_n \: E^n/M^n \to F^n/N^n.$
    \end{itemize}
   Then, $f_n=f'_n$ for all $n$.
\end{lem}

\begin{proof}
 The proof is similar to the proof of the previous lemma.
 We set $h_n:=f_n-f'_n$. By the first condition, $f_n$
 and $f'_n$ coincide on $M^n$. They also coincide on
 $d(M^{n-1})$ by the previous lemma. So, $h_n$ factors through
 $E^n/(M^n+dM^{n-1})=H^{0}(\,^{n-1}\bbX)$.

 By the second condition, $h_n$ factors through $N^n$.
 It follows from the previous lemma that $d\circ h_n$
 factors through $N^{n+1}$. That is, $h^n$ factors through
 $N^n\cap d^{-1}(N^{n+1})=H^{-1}(\,^{n}\bbX)$.

  Putting the above information together, we see that $h_n$ factors
  through a map $H^{0}(\,^{n-1}\bbX) \to H^{-1}(\,^{n}\bbX)$.
  By the properties of a torsion pair, $h_n$ is necessarily
  the zero map.
\end{proof}

We finally come to the proof of Theorem \ref{T:main}.

\begin{proof}[Proof of Theorem \ref{T:main}] \label{Page:proof}
  We define the inverse $G \: \Ch(\A,\mcT,\mcF) \to \Chst(\B)$
  of $\Tot$ as follows.
   Given $(\Eb,\Mb) \in \Ch(\A,\mcT,\mcF)$, we
   define $G(\Eb,\Mb) \in \Chst(\B)$ to be the
   complex whose $n^{th}$ term is
     $$^n\bbX:=[M^{n} \llra{p_{n+1}\circ \delta} E^{n+1}/M^{n+1}].$$
   Here, $p_{n+1} \: E^{n+1} \to E^{n+1}/M^{n+1}$ stands for  the projection
   map. The differential $d \:\,^{n-1}\bbX \to\,^n\bbX$  is the morphism
    defined by $E^n$. More precisely, it is given by the
    diagram
       $$\xymatrix@C=8pt@R=10pt@M=6pt{ M^{n-1} \ar[rd]^{\delta}  \ar[dd]
                            & & M^{n} \ar[ld] \ar[dd]  \\
        & E^n \ar[ld]_(0.4){p_n} \ar[rd]^(0.38){p_{n+1}\circ \delta}  & \\
                  E^{n}/M^{n} & & E^{n+1}/M^{n+1}       }$$
    The arguments of this and the previous subsection
    can be reversed, in a trivial manner, to show that
    $G$ is an inverse to $\Tot$.
\end{proof}

We are  not quite done yet with the proof of Theorem \ref{T:main},
because we have to show that the functors $\Tot$ and $G$ respect
$\sis$ and $\qis$. We do this in the next section.

\section{Effect of  $\Tot$ on  derived
categories}{\label{S:derived}}

In this section we study the effect of $\Tot$ on various cohomology
groups and complete the proof of Theorem \ref{T:main} by showing
that $\Tot$ respects $\sis$ and $\qis$. As a corollary of this, we
reprove a theorem of Happel-Reiten-Smal\o\ asserting that in the
case where the torsion pair is tilting or cotilting there is an
equivalence of derived categories $\mcD(\B) \risom \mcD(\A)$; see
Theorem \ref{T:HRS}.

\subsection{Effect of  $\Tot$ on cohomologies}{\label{SS:more}}

Let $\bfX$ be a complex in $\B$. We denote the $n^{th}$ cohomology
of $\bfX$ by
   $$\bbH^n(\bfX):=\im \,^{n-1}d/\ker\,^n d.$$
(Not to be confused with hypercohomology.) The rest of the notation
appearing in the next proposition have been introduced in
$\S$\ref{SS:deccoh}

\begin{prop}{\label{P:cohisom}}
   For every $\bfX$ in $\Ch(\B)$, we have  natural isomorphisms
   $$H^{-1,n}(\Tot(\bfX))\cong H^{-1}(\,^n\bbX),$$
   $$H^{0,n}(\Tot(\bfX))\cong H^{0}(\,^n\bbX),$$
   $$\bbH^n(\Tot(\bfX))\cong \bbH^n(\bfX).$$
\end{prop}

\begin{proof}
  The first two isomorphisms follow from the definition of $\Tot$.
  The last one is simply a rephrasing of Corollary \ref{C:cohomology}.
\end{proof}

\begin{cor}{\label{C:multsys}}
  The functor $\Tot\: \Chst(\B) \to \Ch(\A,\mcT,\mcF)$
  maps $\sis, \qis \subset \Chst(\B)$
  isomorphically to $\sis, \qis \subset \Ch(\A,\mcT,\mcF)$.
  In particular, the functor $\Tot$ preserves, and reflects,
  quasi-isomorphisms.
\end{cor}

\begin{proof}
  Immediate.
\end{proof}

\subsection{Effect of $\Tot$ on  derived categories}{\label{SS:derived1}}

The shift functor on $\Ch(\A,\mcT,\mcF)$, $(\Eb, \Mb) \mapsto
(\Eb[1], \Mb[1])$, makes the localized category $\qis^{-1}
\Ch(\A,\mcT,\mcF)$ into a triangulated category. We have the
following.

\begin{thm}{\label{T:derived}}
 The functor $\Tot \: \Chst(\B) \to \Ch(\A,\mcT,\mcF)$
 induces a triangle equivalence
  $$\Tot \: \mcD(\B) \risom \qis^{-1} \Ch(\A,\mcT,\mcF).$$
 In particular, we have equivalences of  bounded derived categories
   $$\Tot \: \mcD^*(\B) \risom \qis^{-1} \Ch^*(\A,\mcT,\mcF),$$
 where $ *=-,+,b$.
\end{thm}

\begin{proof}
Follows from Proposition \ref{P:derived} and Corollary
\ref{C:multsys}. Note that $\Tot$ respects any kind of boundedness.
\end{proof}

\subsection{The forgetful functor $\Phi \: \mcD (\B) \to \mcD(\A)$}
{\label{SS:forgetful}}

We have a forgetful triangle functor
    $$\Phi  \: \qis^{-1} \Ch(\A,\mcT,\mcF) \to \mcD(\A),$$
       $$(\Eb, \Mb) \mapsto \Eb.$$
By Theorem \ref{T:derived}, this induces a triangle functor
  $$\Phi\circ\Tot \: \mcD(\B) \to \mcD(\A).$$
Furthermore, by Proposition \ref{P:cohisom}, the following diagram
commutes:
  $$\xymatrix@C=24pt@R=-2pt@M=6pt{ \mcD(\B)
       \ar[dd]_{\Phi\circ\Tot} \ar[rd]^{\bbH} & \\ & \B. \\
                              \mcD(\A) \ar[ru]_{\bbH} &   }$$
Here, the upper $\bbH$ stands for the usual cohomology of chain
complexes, and the lower $\bbH$ is the cohomological functor defined
in $\S$\ref{SS:cohfunc}. The following is immediate.

\begin{prop}
  The functor $\Phi\circ\Tot \: \mcD(\B) \to \mcD(\A)$ reflects
  isomorphisms.
\end{prop}

\subsection{The equivalence $\mcD (\B) \risom  \mcD(\A)$}
{\label{SS:forgetfulequi}}

It is a theorem of Happel-Reiten-Smal\o\ that, in the case where
$(\mcT,\mcF)$ is either tilting or cotilting, there is an
equivalence   of bounded derived categories $\mcD^b(\B) \to
\mcD^b(\A)$.\footnote{In \cite{HaReSm} they actually assume
existence of enough injectives/projectives, but this is not
necessary.}

We reprove  this result using our approach. Recall that a torsion
theory $(\mcT,\mcF)$ is called {\em cotilting} if for every $A \in
\A$ there exists an epimorphism $F \to A$ with $F \in \mcF$. The
torsion theory is called {\em tilting} if for every $A \in \A$ there
exists a monomorphism $A \to T$ with $T \in \mcT$.

Before proving the theorem, we prove a lemma.

\begin{lem}{\label{L:cotilting}}
  Let $(\Eb,\Mb) \in \Ch(\A,\mcT,\mcF)$, and let
  $f \:\Fb \to \Eb$ be a  morphism of complexes.
  Assume that for every $n$, $f^n \: F^n \to E^n$ is an epimorphism
  with kernel in $\mcF$.
  Then the induced decoration $f^{-1}(\Mb)$ on $\Fb$ is compatible
  with $(\mcT,\mcF)$.
\end{lem}

\begin{proof}
  Since $H^{-1,n}(\Fb,f^{-1}(\Mb))$ is an extension
  of $H^{-1,n}(\Eb,\Mb)$ with $\ker f^n$,  it belongs to $\mcF$
  by Lemma \ref{L:torsion3}.
  It is easy to see that $H^{0,n}(\Fb,f^{-1}(\Mb))\cong
   H^{0,n}(\Eb,\Mb)$,
   so it is in $\mcT$.
\end{proof}

\begin{thm}[Happel-Reiten-Smal\o]{\label{T:HRS}}
  Assume $(\mcT,\mcF)$ is cotilting (respectively, tilting). Then
  the functor
    $$\Phi\circ\Tot \: \mcD^*(\B) \to \mcD^*(\A), \ \ *=-,b \ \
             \ \text{(respectively, $*=+,b$)}$$
  defined in $\S$\ref{SS:forgetful} is an equivalence of triangulated
  categories.
\end{thm}

\begin{proof} We prove the claim in the cotilting case, so let $*=-$
or $*=b$. By Theorem \ref{T:derived}, it is enough to show that the
forgetful functor $\Phi \: \qis^{-1} \Ch(\A,\mcT,\mcF) \to \mcD(\A)$
is an equivalence.

  Let $\sfFF^*(\A,\mcT,\mcF) \subset \sfF^*(\A,\mcT,\mcF)$ be as in
  $\S$\ref{SS:homological}. Let $\sfF^*(\A)$ be
  the full subcategory of the homotopy category $\sfK^*(\A)$
  of chain complexes in
  $\A$ consisting of complexes whose terms are in $\mcF$. Clearly
  $\Phi$ induces an equivalence $\sfFF^*(\A,\mcT,\mcF) \risom
  \sfF^*(\A)$ of triangulated categories, so, in particular, it
  induces an equivalence of the
  localized categories $\qis^{-1}\sfFF^*(\A,\mcT,\mcF) \risom
  \qis^{-1}\sfF^*(\A)$. To prove the result, we show that
  the  functors
    $$\qis^{-1}\sfFF^*(\A,\mcT,\mcF) \to \qis^{-1}\Ch^*(\A,\mcT,\mcF)
     \ \ \ \text{and} \ \ \  \qis^{-1}\sfF^*(\A) \to \mcD^*(\A),$$
  induced by the corresponding inclusion
  maps, are equivalences of triangulated categories.

  Let us prove the equivalence on the left.
  By Lemma \ref{L:full}, it is enough to show that
  $\Xi \: \qis^{-1}\sfF^*(\A,\mcT,\mcF) \to \qis^{-1}\Ch^*(\A,\mcT,\mcF)$ is
  an equivalence. We will prove this using the first case of Lemma
  \ref{L:equivalence}.

  Let $(\Eb,\Mb) \in \Ch^*(\A,\mcT,\mcF)$. Since the torsion theory
  is cotilting, we can find a complex $\Fb$ with terms all in
  $\mcF$, together with an epimorphic quasi-isomorphism $f \: \Fb
  \to \Eb$. It follows from Lemma \ref{L:cotilting} that
  $(\Fb,f^{-1}(\Mb))$ is in $\sfF^*(\A,\mcT,\mcF)$. The
  quasi-isomorphism $(\Fb,f^{-1}(\Mb)) \to (\Eb,\Mb)$ guarantees
  that Lemma \ref{L:equivalence} applies. This proves that $\Xi$ is
  an equivalence.

  The proof that $\qis^{-1}\sfF^*(\A) \to \mcD^*(\A)$ is an equivalence
  is entirely similar.
\end{proof}

\section{Iterating the construction for $\B$}{\label{S:iterate}}

The abelian category $\B$ comes with  a torsion pair
$(\mcT',\mcF')$, where $\mcT'=\mcF[1]$ and $\mcF'=\mcT$. We describe
the abelian category $\C$ obtained by applying the tilting procedure
to this torsion pair.

\subsection{Objects of $\C$}{\label{SS:objectsC}}

It follows from the description of kernels and cokernels in
$\S$\ref{SS:kernel} that an object in $\C$ is a diagram

  $$\xymatrix@C=8pt@R=6pt@M=6pt{ X^{-1} \ar@{^(->}[rd]^{\kappa}
  \ar[dd]
                            & & Y^{-1} \ar[ld]_{\iota} \ar[dd]  \\
                  & E \ar[ld]^{\sigma} \ar@{->>}[rd]_{\rho}  & \\
                  X^0 & & Y^0       }$$
where $\kappa$ is a monomorphism and $\rho$ is an epimorphism.
Equivalently, an object in $\C$ is a 4-tuple $[K_1\subseteq
K_2\subseteq E \supseteq M]$ satisfying the following conditions:
\begin{itemize}
  \item[$\blacktriangleright$] \ $K_2 \cap M \in \mcF$,

  \item[$\blacktriangleright$] \ $E/(K_1+M) \in \mcT$.
\end{itemize}
(Take $K_1=\kappa(X^{-1})$, $K_2=\ker\rho$, and $M=\iota (Y^{-1})$.)

\subsection{Morphisms of $\C$}{\label{SS:morphismsC}}
A {\em strict} morphism $[K_1\subseteq K_2\subseteq E \supseteq M]
\to [K'_1\subseteq K'_2\subseteq E' \supseteq M']$ is a morphism $f
\: E\to E'$ in $\A$ respecting  the three subobjects. Let us denote
the category of such 4-tuples and strict morphisms between them by
$\C^{st}$. Note that $\C^{st}$ is naturally identified with a full
subcategory of $\Ch^{st}(\B)$. We denote  $\qis \cap \C^{st}$ by the
same notation $\qis$; see Section \ref{S:complexesstrict} for
notation. In other words, a  morphism in $\C^{st}$ is in $\qis$ if
it induces an isomorphism $K_1/K_2 \risom K'_1/K'_2$.

\begin{prop}{\label{P:C}}
The inclusion $i \: \C^{st} \to \C$ induces an additive equivalence
of abelian categories $\qis^{-1}\C^{st} \risom \C$. Furthermore, the
functor
         $$H \: \C^{st} \to \A,$$
            $$[K_1\subseteq K_2\subseteq E \supseteq M] \mapsto
            K_2/K_1$$
fits into the following commutative (up to a canonical natural
transformation) diagram:
   $$\xymatrix@C=16pt@R=14pt@M=6pt{\C^{st} \ar[r]^{H} \ar[d]_(0.4)i & \A \\
                            \C\cong \qis^{-1}\C^{st} \ar[ru] & }$$
 (Also, note that $\qis^{-1} \subset \C^{st}$ is exactly
 the class of morphisms $\C^{st}$ that map to isomorphisms in $\A$ under
 $H$.)
\end{prop}

\begin{proof}
 That  $i \: \C^{st} \to \C$ induces an equivalence of
 additive categories $\qis^{-1}\C^{st} \risom \C$ follows from Lemma
 \ref{L:localizing}.  The existence of the commutative
 triangle is trivial.
\end{proof}

The functor $\C \to \A$, which we will, by abuse of notation, denote
by $H$, is known to be an equivalence of abelian categories whenever
$(\mcT,\mcF)$ is either tilting or cotilting. Let us give a quick
proof of this.

\begin{thm}[Happel-Reiten-Smal\o]{\label{T:HRS2}}
   Assume $(\mcT,\mcF)$ is either tilting or cotilting.
   Then, the functor $H \: \C \to \A$ defined above
   is an additive equivalence of
   categories.
\end{thm}

\begin{proof}
   Assume $(\mcT,\mcF)$ is tilting.
   We define the inverse functor $Q\: \A \to  \C$ as
   follows. For every $A \in \A$ choose a monomorphism
   $i_A \: A \hookrightarrow T$ with $T \in \mcT$.
   The effect of $Q$ on objects is $Q(A):=[0\subseteq A \subseteq T
   \supseteq 0]$. To define the effect of $Q$ on morphisms, let
   $f \: A \to A'$ be  a morphism
    in $\A$, and let $Q(A')=[0\subseteq A' \subseteq T' \supseteq 0]$.
    Set $X=[0\subseteq A  \subseteq T\oplus T' \supseteq 0]$, where
    $A \hookrightarrow  T\oplus T'$ is the map $(i_A,i_{A'}\circ f)$.
    We define $Q(f)$ to be the roof
      $$\xymatrix@C=10pt@R=8pt@M=6pt{
        & X\ar[ld]^(0.35){\sim}_{(\id, \pr_1)}
                           \ar[rd]^{(f,\pr_2)} & \\
                                     Q(A) && Q(A'). }$$
    It is easy to check that $Q$ is a well-defined functor and that
    it is an inverse equivalence to $H$.

    Proof in the cotilting case  is similar. The functor $Q' \: \A
    \ \to \C$ sends an object $A \in \A$ to $[\ker{p}\subseteq F
    \subseteq F \supseteq F]$, where $p \: F \to A$ is a choice of an
    epimorphism with $F \in \mcF$.
\end{proof}

\section{The DG structure}{\label{S:DG}}

The category $\Ch(\B)$ is naturally a DG category, and so is
$\Ch(\A,\mcT,\mcF)$, as we will see shortly. So one would like to
strengthen Theorem \ref{T:derived} to a statement about DG
categories. We will do that in this and the next section.

We will assume that $\A$ is a $K$-linear abelian category, where $K$
is a commutative  ring.

\subsection{The symmetric monoidal category $\Dec(K)$}{\label{SS:Dec}}

 Let $\Dec(K)$ denote the category of decorated complexes
 of $K$-modules. This is a symmetric monoidal category.
 Given two decorated complexes $(\Eb,\Mb)$
 and $(\Fb,\Nb)$, their tensor product is the
  complex $\Eb\otimes \Fb$, decorated with
  the image of
 $(\Eb\otimes \Nb)\oplus (\Mb\otimes \Fb)$.

 We discuss the inner homs in $\S$\ref{SS:Enrich1}.

\subsection{The $\Dec(K)$ enrichment of $\Dec(\A)$}{\label{SS:Enrich1}}

We will introduce a  $\Dec(K)$ enrichment of $\Dec(\A)$, which we
denote by $\Decfrak(\A)$. This is stronger than a DG structure.

Given two objects $(\Eb,\Mb)$ and $(\Fb,\Nb)$ in $\Dec(\A)$, we
define the $K$-complex
      $$\homfh\big((\Eb,\Mb),(\Fb,\Nb)\big)$$
to be the subcomplex of the usual mapping complex
$\homf\big(\Eb,\Fb)$ consisting of maps satisfying certain
compatibility with respect to $\Mb$ and $\Nb$. More precisely, an
element in the $k^{th}$ term of the complex
$\homfh\big((\Eb,\Mb),(\Fb,\Nb)\big)$ is a collection of morphisms
$g_n \: E^n\to F^{n+k}$, $n\in\bbZ$, such that the following
conditions are satisfied:

\label{Page:hom}

\begin{itemize}
  \item[$\mathbf{1}$)] For every $n$, $g_n(M^n)\subseteq N^{n+k}$.

  \item[$\mathbf{2}$)] For every $n$, $\delta_F\circ g_n-(-1)^kg_{n+1}\circ
  \delta_E$ maps $M^n$ to $N^{n+k+1}$. Equivalently, the following
  diagram commutes:
           $$\xymatrix@C=20pt@R=14pt@M=6pt{
                              M^n \ar[d]_{g_n} \ar[r]^{\delta_E}
                                        & E^{n+1} \ar[d]^{(-1)^k g_{n+1}} \\
                  F_{n+k} \ar[r]_(0.3){\delta_F} & F^{n+k+1}/N^{n+k+1} }$$
\end{itemize}

The differential $d$ on $\homfh\big((\Eb,\Mb),(\Fb,\Nb)\big)$ is
defined as usual:
   $$d \:
   \homfh\big((\Eb,\Mb),(\Fb,\Nb)\big)^k
    \to
    \homfh\big((\Eb,\Mb),(\Fb,\Nb)\big)^{k+1},$$
  $$\{g_n\}_{n\in\bbZ} \mapsto \{\delta_F\circ g_n-(-1)^kg_{n+1}\circ
  \delta_E\}_{n\in\bbZ}.$$
It is not hard to verify that the sequence $\{\delta_F\circ
g_n-(-1)^kg_{n+1}\circ
  \delta_E\}_{n\in\bbZ}$ also satisfies ($\mathbf{1}$) and ($\mathbf{2}$).

There is a natural decoration
$$\homfone\big((\Eb,\Mb),(\Fb,\Nb)\big) \subseteq
     \homfh\big((\Eb,\Mb),(\Fb,\Nb)\big)$$
on $\homfh\big((\Eb,\Mb),(\Fb,\Nb)\big)$ whose $k^{th}$ term is, by
definition, the $K$-submodule of
$\homfh\big((\Eb,\Mb),(\Fb,\Nb)\big)^k$
 consisting of those sequences
$\{g_n\}_{n\in\bbZ}$ satisfying the following axiom:

\begin{itemize}
  \item[$\blacktriangleright$] For every $n$, $g_n(M^n)=0$ and
   $g_n(E^n) \subseteq N^n$.  Note that this condition implies
   ($\mathbf{1}$) and  ($\mathbf{2}$).
\end{itemize}

\subsection{The $\Dec(K)$ enrichment of $\Ch(\A,\mcT,\mcF)$}
{\label{SS:Enrich2}}

The category  $\Ch(\A,\mcT,\mcF)$, being a full subcategory of
$\Dec(\A)$, inherits a $\Dec(K)$ enrichment, which we denote by
$\Chfrak(\A,\mcT,\mcF)$. This $\Dec(K)$ enrichment has a special
feature that is described in the next proposition. We simplify the
notation by denoting $\homfone\big((\Eb,\Mb),(\Fb,\Nb)\big)$ by
$\homfone$ and $\homfh\big((\Eb,\Mb),(\Fb,\Nb)\big)$ by $\homfh$.

\begin{prop}{\label{P:enrich}}
  For every $k$, $\homfone^{k}\cap d^{-1}(\homfone^{k+1})=0$.
  That is $H^{-1,k}(\homfh,\homfone)=0$, for all $k$; see
  $\S$\ref{SS:deccoh} for notation.
\end{prop}

\begin{proof}
  Let $\{g_n\}_{n\in\bbZ}$ be in
  $\homfone^{k}\cap d^{-1}(\homfone^{k+1})$.
  Then the following are true for all $n$:
   \begin{itemize}
    \item[i)] $g_n(M^n)=0$,
    \item[ii)] $g_n(E^n) \subseteq N^{n+k}$,
    \item[iii)] $\delta_F\circ g_n-(-1)^kg_{n+1}\circ
         \delta_E$ vanishes on $M^n$, and
    \item[iv)]  $\delta_F\circ g_n-(-1)^kg_{n+1}\circ
         \delta_E$ maps $E^n$ to $N^{n+k+1}$.
   \end{itemize}
  It follows from (i) and (iii), both with $n-1$, that $g_n$
  vanishes on $\delta_E(M^{n-1})$. This implies that $g_n$ factors
  through
        $$H^{0,n-1}(\Eb,\Mb)=E^n/(M^n+\delta_E(M^{n-1}).$$
   Also, it follows from (iv), with $n$, and (ii),
  with $n+1$, that $g_n(E^n) \subseteq \delta_F^{-1}(N^{n+k+1})$.
  Therefore, $g_n$ factors through
  $$H^{-1,n+k}(\Fb,\Nb)=N^{n+k}\cap \delta_F^{-1}(N^{n+k+1}).$$
  Put
  together, we see that $g_n$ factors through a map
       $$H^{0,n-1}(\Eb,\Mb) \to H^{-1,n+k}(\Fb,\Nb).$$
  Since the left hand side belongs to $\mcT$ and the right hand side
  belongs to $\mcF$, this map has to be zero. Hence, $g_n=0$.
\end{proof}

\begin{cor}{\label{C:enrich}}
  The decorated complex $(\homfh,\homfone)$ is naturally
  quasi-isomorphic to the  complex
  $\mcH^0(\homfh,\homfone)[-1]$ endowed with the zero decoration; see
  $\S$\ref{SS:deccoh} for notation.
\end{cor}

This corollary means that $\Chfrak(\A,\mcT,\mcF)$ should simply be
thought of as a DG category, with the hom complexes being
$\mcH^0(\homfh,\homfone)[-1]$.

\vspace{0.1in} \noindent{\bf Abuse of notation.} We denote the DG
category whose objects are the ones of $\Ch(\A,\mcT,\mcF)$ and whose
hom complexes are $\mcH^0(\homfh,\homfone)[-1]$ also by
$\Chfrak(\A,\mcT,\mcF)$.

\subsection{DG structures on $\Ch(\B)$ and $\Chst(\B)$}{\label{SS:DGChst}}

Being the category of chain complexes in a $K$-linear abelian
category, $\Ch(\B)$ carries  a natural $K$-linear DG structure,
which we denote by $\Chfrak(\B)$.
 The DG
 structure of $\Chfrak(\B)$
induces a DG structure on $\Chst(\B)$ as well, which we denote by
$\Chstfrak(\B)$.  By definition, an element in
$\homf_{\Chstfrak(\B)}(\bfX,\bfY)^k\subseteq\homf_{\Chfrak(\B)}(\bfX,\bfY)^k$
is a sequence of morphisms $h_n \: \,^n\bbX \to \,^{n+k}\bbY$, $n
\in \bbZ$, satisfying the following conditions:
\begin{itemize}
  \item[$\blacktriangleright$] Every $h_n$ is strict;

  \item[$\blacktriangleright$] The dotted arrow in the following
  diagram can be filled:
       $$\xymatrix@=10pt@M=8pt@C=14pt{^{n-1}X^0  \ar[d]_{h_{n-1}^0} & E^n\ar[l]
                      \ar@{..>}[d]  &  \,^{n}X^{-1} \ar[d]^{h_{n}^{-1}} \ar[l] \\
                              ^{n-1}Y^0  & F^n \ar[l] &  \,^{n}Y^{-1} \ar[l] }$$
  Here, $^n\bbX=[\,^{n}X^{-1} \to \,^{n}X^{0}]$,
  $^n\bbY=[\,^{n}Y^{-1} \to \,^{n}Y^{0}]$,
  and $E^n$ and $F^n$ are the extensions
  that define the morphisms
  $d \:\,^{n-1}\bbX \to\,^n\bbX$ and $d \:\,^{n-1}\bbY \to\,^n\bbY$.
\end{itemize}

That $\homf_{\Chstfrak(\B)}(\bfX,\bfY)$ is indeed a subcomplex of
$\homf_{\Chfrak(\B)}(\bfX,\bfY)$ is an easy exercise (it also
follows from the proof of Proposition \ref{P:DGeq} in the next
section).

\subsection{Remark: decorations of length $l$}{\label{SS:remark}}

 A decorated complex can be thought of as a complex each of whose
 terms is equipped with a length one filtration (but the
 differentials do not necessarily respect the filtrations). We can
 drop the requirement on the length of the filtration and consider
 complexes with {\em length $l$ decorations}, $0\leq l\leq \infty$.
 It turns out that the category $\Dec^l(K)$ of chain complexes of
 $K$-modules with length $l$ decorations has a natural closed
 monoidal structure, and for every $K$-linear category $\A$, the
 category $\Dec^l(\A)$ of chain complexes in $\A$ with length $l$
 decorations is naturally enriched over $\Dec^l(K)$. The homological
 algebra of decorated complexes, as described in
 $\S$\ref{SS:homological},
  carries over to $\Dec^l(K)$.

\section{The derived DG equivalence between
   $\Chfrak^b(\B)$ and $\Chfrak^b(\A,\mcT,\mcF)$}{\label{S:DGEquiv}}

 The  functor $G \: \Ch(\A,\mcT,\mcF) \to
\Ch(\B)$    defined in the proof of Theorem \ref{T:main} can be
enriched to a DG functor $\mfG \: \Chfrak(\A,\mcT,\mcF) \to
\Chstfrak(\B)$ in the obvious way; see the proof of Proposition
\ref{P:DGeq} below. We will use this to construct a DG equivalence
between $\Chfrak^b(\B)$ and $\Chfrak^b(\A,\mcT,\mcF)$ strengthening
the derived equivalence of Theorem \ref{T:derived}; see Theorem
\ref{T:derequiv} below.

 \begin{prop}{\label{P:DGeq}}
   The functor $\mfG$ induces a DG equivalence
     $$\mfG \: \Chfrak(\A,\mcT,\mcF) \to \Chstfrak(\B).$$
\end{prop}

Before proving the proposition, we prove a lemma.

   \begin{lem}{\label{L:null}}
    A strict morphism $(f^{-1},f^0) \: [X^{-1} \torel{d_X} X^0]
    \to [Y^{-1} \torel{d_Y} Y^0]$
    in $\B$ is the zero morphism if and only if
     there exists $s \: X^0 \to
    Y^{-1}$ such that $f^0=d_Y\circ s$ and $f^{-1}=s\circ d_X$,
    that is, $(f^{-1},f^0)$ is null-homotopic in $\Ch(\A)$.
   \end{lem}

\begin{proof}
  Use the description of strict morphisms in
  $\S$\ref{SS:definition}.
\end{proof}

\begin{proof}[Proof of Proposition \ref{P:DGeq}]
 Let $(\Eb,\Mb)$ and $(\Fb,\Nb)$
 be objects in $\Chfrak(\A,\mcT,\mcF)$, and set
 $\bfX=G(\Eb,\Mb)$ and $\bfY=G(\Fb,\Nb)$,
 as in the proof of Theorem \ref{T:main} (page
 \pageref{Page:proof}).

  Consider an element $\gamma$ in
$\homfh\big((\Eb,\Mb),(\Fb,\Nb)\big)^k$ given by the sequence $g_n
\: E^n\to F^{n+k}$, $n\in\bbZ$. By ($\mathbf{1}$) of page
\pageref{Page:hom}, each $g_n$ induces
 maps $\underline{g}_n\: M^n \to N^{n+k}$ and
 $\overline{g}_n\: E^{n}/M^{n} \to
 F^{n+k}/N^{n+k}$. Condition ($\mathbf{2}$) guarantees that the map
$$\mfG(\gamma)_n:=\big((-1)^k\underline{g}_n,\overline{g}_{n+1}\big)
     \: [M^n \to E^{n+1}/M^{n+1}] \to
                        [N^{n+k}\to F^{n+k+1}/N^{n+k+1}]$$
 is a morphism in $\B$. (Recall that, by definition,
 $^n\bbX=[M^n \to E^{n+1}/M^{n+1}]$ and  $\,^{n+k}\bbY=[N^{n+k}\to
 F^{n+k+1}/N^{n+k+1}]$.)
We define $\mfG(\gamma)$ to be the sequence $\mfG(\gamma)_n$,
  $n\in\bbZ$. It is easy to see that the map of $K$-modules
     $$\homfh\big((\Eb,\Mb),(\Fb,\Nb)\big)^k
                       \to \homf_{\Chstfrak(\B)}(\bfX,\bfY)^{k}$$
             $$ \qquad \gamma \mapsto \mfG(\gamma) $$
 is surjective. We claim that its kernel is equal to
   $$\homfone\big((\Eb,\Mb),(\Fb,\Nb)\big)^k
   + d
   \homfone\big((\Eb,\Mb),(\Fb,\Nb)\big)^{k-1}.$$
  By Lemma \ref{L:null}, applied to $\mfG(\gamma)_n$,  a sequence
  $g_n \: E^n\to F^{n+k}$, $n\in\bbZ$, is in the kernel of the
  above map
  if and only if there is a sequence $s_{n} \: E^{n}\to F^{n+k-1}$
   such that for all $n$
     \begin{itemize}
         \item[i)] \ $s_n(M^{n})=0$,
         \item[ii)] \ $s_n(E^{n})\subseteq N^{n+k-1}$,
         \item[iii)] \ $g_n-(-1)^ks_{n+1}\circ\delta_E $ vanishes
         on $M^n$, and
         \item[iv)] \ $g_n-\delta_F\circ s_n$ maps $E^n$ to
         $N^{n+k}$.

     \end{itemize}
      We clearly have
         $\sigma:=\{s_n\}_{n\in\bbZ}\in \homfone
             \big((\Eb,\Mb),(\Fb,\Nb)\big)^{k-1}.$
  It is also straightforward from conditions (i)-(iv) above that
         $$\gamma-d(\sigma)=\{g_n-\delta_F\circ
     s_n+(-1)^{k-1}s_{n+1}\circ\delta_E\}_{n\in\bbZ} \in
                    \homfone\big((\Eb,\Mb), (\Fb,\Nb)\big)^{k}.$$
  This proves our claim about the kernel.

  Abbreviating
  $\homfh\big((\Eb,\Mb),(\Fb,\Nb)\big)$
  and
  $\homfone\big((\Eb,\Mb),(\Fb,\Nb)\big)$
   to $\homfh$ and $\homfone$, respectively, we  summarize what
   we have proved by saying that $\homfh^k \to
   \homf_{\Chstfrak(\B)}(\bfX,\bfY)^{k}$ is a surjective map of
   $K$-modules
   whose kernel is $\homfone^k+d\homfone^{k-1}$. Therefore,
   we have an induced isomorphism
      $$\mcH^0(\homfh,\homfone)[-1]^k=\homfh^k/(\homfone^k+d\homfone^{k-1})
       \risom \homf_{\Chstfrak(\B)}(\bfX,\bfY)^{k}.$$
   This is exactly what we wanted to prove.
\end{proof}

\begin{rem}
The DG equivalence of the previous lemma is a DG equivalence in a
strong sense: it induces {\em isomorphisms} on hom complexes.
\end{rem}

It is easy to see that the functor $\Tot$ can also be enriched to a
DG equivalence $\Totfrak \: \Chstfrak(\B) \to
\Chfrak(\A,\mcT,\mcF)$, and that $\Totfrak$ and $\mfG$ are inverse
to each other.

Of course, we expect that applying $Z_0$ to the above enrichments
give us back the old categories.

\begin{prop}
  We have equivalences 
$$
\begin{array}{rcl}
     Z_0\mathfrak{Ch}(\B)&\cong&\Ch(\B),\\
     Z_0\mathfrak{Ch}^{st}(\B)&\cong&\Chst(\B),\\
     Z_0\Chfrak(\A,\mcT,\mcF)&\cong&\Ch(\A,\mcT,\mcF),\\
     H_0\Chfrak(\A,\mcT,\mcF)&\cong&\sfK(\A,\mcT,\mcF).
\end{array}$$
(See $\S$\ref{SS:homological} for notation.)
\end{prop}

\subsection{Semi-injective and semi-projective objects in
$\B$}{\label{SS:semi}} Before we prove the derived equivalence of
$\Chfrak(\B)$ and $\Chfrak(\A,\mcT,\mcF)$ we need some definitions.

We say that an object $\bbX=[X^{-1} \to X^0]$  in $\B$ is {\em
semi-injective} (respectively, {\em semi-projective}), if $X^{-1}$
an injective object in $\A$ (respectively, if $X^0$ a projective
object in $\A$). Note that these notions are {\em not} invariant
under isomorphism.

\begin{prop}{\label{P:injproj}}
  Let $\bfX$ and $\bfY$ be complexes in $\Ch(\B)$. Assume
  either $\bfX$ is a complex of semi-projective objects, or $\bfY$
  is a complex of semi-injective objects. Then
  $$\homf_{\Chstfrak}(\bfX,\bfY)\hookrightarrow
                                          \homf_{\Chfrak}(\bfX,\bfY) $$
is an isomorphism. (We do not need any boundedness conditions on
$\bfX$ or $\bfY$).
\end{prop}

\begin{proof}
  This follows from Lemma \ref{L:split}.
\end{proof}

\begin{lem}{\label{L:enough}}
  If $\A$ has enough injectives, then for every object
  $\bbX$ in $\B$ there exists a semi-injective object $\mathbb{I}$
  and  a strict isomorphism $\bbX \to \mathbb{I}$.
  If $\A$ has enough projectives, then for every object
  $\bbX$ in $\B$ there exists a semi-projective object $\mathbb{P}$
  and  a strict isomorphism $ \mathbb{P} \to \bbX$.
\end{lem}

\begin{proof}
  Assume $\A$ has enough injectives. Take a monomorphism
  $X^{-1} \hookrightarrow I$ into an injective $I$,
  an set $\mathbb{I}=[I \to X^0\oux{}{X^{-1}}I]$.
  Similarly, if $\A$ has enough projectives,
  take an epimorphism $P \twoheadrightarrow X^0$ from a
  projective $P$, and set $\mathbb{P}=[P\oux{X^0}{}X^{-1} \to P]$.
\end{proof}

\begin{cor}{\label{C:resolutions}}
 Let $\bfX$ be a complex in $\B$. If $\A$ has enough injectives,
 then there is a complex $\mathbf{I}$ of semi-injective objects in
 $\B$ and a strict isomorphism $\bfX \to \mathbf{I}$.
 If $\A$ has enough projectives,
 then there is a complex $\mathbf{P}$ of semi-projective objects in
 $\B$ and a strict isomorphism $\mathbf{P} \to\bfX$. (No boundedness
 conditions needed.)
\end{cor}

\subsection{Derived equivalence of  $\Chfrak^b(\B)$ and
$\Chfrak^b(\A,\mcT,\mcF)$  }{\label{SS:derivedeq1}}

In this subsection we will need   $\A$ to have  either enough
injectives or enough projectives. Let us say that $\A$ has enough
projectives. For $\bfX,\bfY \in \Chstfrak(\B)$ we define
   $$\RHom_{\Chstfrak(\B)}(\bfX,\bfY):=
                  \homf_{\Chstfrak(\B)}(\mathbf{P},\bfY)$$
where $\mathbf{P} \to \bfX$ is a semi-projective resolution as in
Corollary \ref{C:resolutions}. If $\mathbf{P}'$ is another
semi-projective resolution for $\bfX$, it follows from Lemma
\ref{L:split} that there is a canonical strict isomorphism
$\mathbf{P}'\risom \mathbf{P}$ over $\bfX$, with a strict inverse
$\mathbf{P}\risom \mathbf{P}'$. Therefore,
$\RHom_{\Chstfrak(\B)}(\bfX,\bfY)$ is well-defined up to a canonical
{\em isomorphism}.

Proposition \ref{P:injproj} implies that there is an isomorphism of
$K$-complexes
  $$\RHom_{\Chstfrak(\B)}(\bfX,\bfY) \cong
  \homf_{\Chfrak(\B)}(\bfX,\bfY).$$
In other words, the inclusion $\Chstfrak(\B) \hookrightarrow
\Chfrak(\B)$ is a ``derived'' equivalence; here ``derived'' refers
to the localizing class $\sis$ and not $\qis$. By Proposition
\ref{P:DGeq}, we find that $\mfG \: \Chfrak(\A,\mcT,\mcF) \to
\Chfrak(\B)$ is a ``derived'' equivalence. Similar discussion is
valid in the case where $\A$ has enough injectives. We summarize
this in the following proposition.

\begin{prop}{\label{P:dereq}}
  Assume that $\A$ has either enough injectives or enough
 projectives. Then, we have a natural ``derived'' equivalence
     $$\Chfrak(\B)\cong\Chfrak(\A,\mcT,\mcF)$$
 of DG categories.
(Here, ``derived'' refers to the localizing class $\sis$.)
\end{prop}

From this we deduce the DG version of Theorem \ref{T:derived}.

\begin{thm}{\label{T:derequiv}}
  Assume that $\A$ has either enough injectives or enough
  projectives. Assume further that $\B$ has enough injectives
  (respectively, enough projectives). Let $*=+,b$ (respectively,
  $*=-,b$.) Then, we have a derived equivalence
  $$\Chfrak^*(\B)\cong\Chfrak^*(\A,\mcT,\mcF)$$
  of DG categories.
  (Here, ``derived'' refers to either of the two localizing class $\sis$ or $\qis$.)
\end{thm}

Some explanation about the meaning of this theorem is perhaps
helpful. First of all, the real interesting case of the theorem is
when the localizing class is $\qis$;  that is what the word
``derived'' is usually associated with.  We know that given $\bfX$
and $\bfY$ in $\Chfrak^*(\B)$, the hom-complex
$\homf_{\Chfrak(\B)}(\bfX,\bfY)$ is not well-behaved with respect to
quasi-isomorphisms. That is why, as in the case of derived
categories, we are more interested in the derived hom-complexes,
namely, the ones obtained by first replacing $\bfX$ and $\bfY$ by an
appropriate  projective or injective resolution, and then taking
$\homf$. What the above theorem is saying is that, the functors
between $\Chfrak^*(\B)$ and $\Chfrak^*(\A,\mcT,\mcF)$ do not
necessarily induce quasi-isomorphisms on the usual hom-complexes
$\homf_{\Chfrak(\B)}(\bfX,\bfY)$, but they do induce
quasi-isomorphisms on the {\em derived} hom-complexes.

\providecommand{\bysame}{\leavevmode\hbox
to3em{\hrulefill}\thinspace}
\providecommand{\MR}{\relax\ifhmode\unskip\space\fi MR }
\providecommand{\MRhref}[2]{%
  \href{http://www.ams.org/mathscinet-getitem?mr=#1}{#2}
} \providecommand{\href}[2]{#2}


\end{document}